\documentclass{amsart}
\usepackage{amsmath,amsfonts,amssymb,amscd,verbatim,multicol}
\usepackage[arrow,matrix,cmtip]{xy}

\usepackage{fullpage}

\newcommand{\ff}{\footnote}

\newcommand{\lann}{\operatorname{l.ann}}
\newcommand{\rann}{\operatorname{r.ann}}

\newcommand{\trdeg}{\operatorname{tr.deg}}

\newcommand{\Div}{\operatorname{Div}}

\newcommand{\bpr}{\begin{proof}}
\newcommand{\epr}{\end{proof}}

\newcommand{\spec}{\operatorname{Spec}}
\newcommand{\ext} {\operatorname{Ext}}

\newcommand{\lra} {\longrightarrow}

\newcommand{\D}{\displaystyle}
\newcommand{\mc}{\mathcal}
\newcommand{\mf}{\mathfrak}

\newcommand{\mb}{\mathbb}

\renewcommand{\hom}{\operatorname{Hom}}

\newcommand{\GK}{\operatorname{GK}}

\newcommand{\wt}{\widetilde}

\newcommand{\tor}{\operatorname{Tor}}

\newcommand{\im}{\operatorname{im}}
\newcommand{\cd}{\operatorname{cd}}

\newcommand{\proj}{\operatorname{proj}}

\newcommand{\rproj}{\operatorname{proj-}}

\newcommand{\coh}{\operatorname{coh}}

\newcommand{\rQcoh}{\operatorname{Qcoh}}

\newcommand{\rGr}{\operatorname{Gr-}\hskip -2pt}
\newcommand{\rgr}{\operatorname{gr-}\hskip -2pt}
\newcommand{\rQgr}{\operatorname{Qgr-}\hskip -2pt}
\newcommand{\rqgr}{\operatorname{qgr-}\hskip -2pt}

\newcommand{\rTors}{\operatorname{Tors-}\hskip -2pt}

\DeclareMathOperator{\Hom}{Hom}

\DeclareMathOperator{\Pic}{Pic}
\DeclareMathOperator{\HB}{H}
\DeclareMathOperator{\HS}{h}
\DeclareMathOperator{\hs}{h}

\newcommand{\dv}{\operatorname{div}}
\newcommand{\ts}{Z}
\newcommand{\on}{\operatorname}
\newcommand{\gra}{\operatorname{gr}}

\numberwithin{equation}{section}
 \theoremstyle{plain}
\newtheorem{theorem}[equation]{Theorem}

\newtheorem{lemma}[equation]{Lemma}

\newtheorem{corollary}[equation]{Corollary}
\newtheorem{proposition}[equation]{Proposition}

\newtheorem{question}[equation]{Question}

\theoremstyle{definition}
\newtheorem{definition}[equation]{Definition}

\newtheorem{remark}[equation]{Remark}

\newtheorem{hypothesis}[equation]{Hypothesis}
\newtheorem{example}[equation]{Example}

\begin{document}

\title{Blowup subalgebras of the Sklyanin algebra}

\author{D. Rogalski}

\address{Department of Mathematics, UCSD, La Jolla, CA 92093-0112, USA. }

\thanks{The author was partially supported by NSF grants DMS-0600834 and DMS-0900981.}
\email{drogalsk@math.ucsd.edu}


\begin{abstract}
We describe some interesting graded rings which are generated by
degree-$3$ elements inside the Sklyanin algebra $S$, and prove that
they have many good properties.  Geometrically, these rings $R$
correspond to blowups of the Sklyanin $\mb{P}^2$ at $7$ or fewer
points.   We show that the rings $R$ are exactly those
degree-$3$-generated subrings of $S$ which are maximal orders in the
quotient ring of the $3$-Veronese of $S$.
\end{abstract}
\maketitle
\tableofcontents
\section{Introduction}
One of the major objectives in the theory of noncommutative
projective geometry is the classification of noncommutative
surfaces.  The goal of this paper is to study, from a ring-theoretic
standpoint, some surfaces which are birational to the generic
noncommutative projective plane, the Sklyanin $\mb{P}^2$.  We first
describe our main results, and then explain in more detail how they
are motivated by the classification project.  In particular, we
discuss at the end of the introduction how our results relate to Van
den Bergh's blowing up construction in \cite{VdB}.

We first review a few standard definitions.  Fix an algebraically
closed field $k$.    In this paper, we are primarily interested in
$\mb{N}$-graded associative $k$-algebras $A = \bigoplus_{n =
0}^{\infty} A_n$ which are connected ($A_0 = k$) and finitely graded
(finitely generated as a $k$-algebra), as well as domains.
The Gelfand-Kirillov (GK) dimension of $A$ can be defined, in this
graded setting, by $\GK(A) = 1 + \limsup_{n \geq 0} \log_n (\dim_k
A_n)$.  We assume that $A$ is a connected finitely graded (cfg)
domain with integer GK-dimension for the rest of this introduction.
Any such $A$ has a \emph{graded quotient ring} $Q_{\rm gr}(A)$
formed by inverting the Ore set of all nonzero homogeneous elements,
as well as the usual Goldie quotient ring $Q(A)$ formed by inverting
all nonzero elements. Domains $A$, $A'$ with $Q = Q(A) = Q(A')$ are
\emph{equivalent orders} if there are nonzero elements $p_1, p_2,
q_1, q_2 \in Q$ such that $p_1 A p_2 \subseteq A'$ and $q_1 A' q_2
\subseteq A$; a \emph{maximal order} is a subring $A$ of $Q$ maximal
under inclusion among orders in an equivalence class.

A construction that plays a fundamental role in our results is that
of a twisted homogeneous coordinate ring.  Such a ring $B(X, \mc{L},
\sigma)$ is built out of the data of a projective $k$-scheme $X$, an
invertible sheaf $\mc{L}$ on $X$, and an automorphism $\sigma: X \to
X$. Putting $\mc{L}_0 = \mc{O}_X$ and $\mc{L}_n = \mc{L} \otimes
\sigma^*(\mc{L}) \otimes \dots \otimes (\sigma^{n-1})^*(\mc{L})$ for
$n \geq 1$, then we define $B = B(X, \mc{L}, \sigma) = \bigoplus_{n
= 0}^{\infty} \HB^0(X, \mc{L}_n)$, with a multiplication defined on
graded pieces by $f \star g = f \cdot (\sigma^m)^*(g)$ for $f \in
B_m, g \in B_n$. Here, $(\sigma^m)^*: \HB^0(X, \mc{L}_n) \to
\HB^0(X, \mc{L}_{n}^{\sigma^m})$ is the map on global sections
induced by pullback by $\sigma^m$, and $\cdot$ indicates the
multiplication map $\HB^0(X, \mc{L}_m) \otimes \HB^0(X,
\mc{L}_n^{\sigma^m}) \to \HB^0(X, \mc{L}_{m+n})$.  We only consider
this construction under an additional condition that $\mc{L}$ is
\emph{$\sigma$-ample}, which ensures that $B$ is noetherian.  See
Section~\ref{background-sec} for the details.

Noncommutative projective geometry might be said to have begun with
a project of Artin and Schelter to classify certain graded algebras
of dimension 3 which are analogs of commutative polynomial rings
(naturally, these are called AS-regular algebras today).  We omit
the technical definition of AS-regular, which is not needed for this
paper. The classification, which was completed with Tate and Van den
Bergh in \cite{ATV1}, \cite{ATV2}, in fact necessitated the
development of the concept of a twisted homogeneous coordinate ring.
We concentrate here on those regular algebras with three degree-1
generators and three quadratic relations, which correspond
geometrically to noncommutative $\mb{P}^2$'s. The classification of
these includes the \emph{linear} regular algebras, which have the
form $B(\mb{P}^2, \mc{O}(1), \sigma)$, but the more interesting
examples are the \emph{elliptic} ones.  Such an elliptic regular
algebra $A$ always has a normal element $g \in A_3$ such that $A/Ag
\cong B(E, \mc{L}, \sigma)$ for some degree $3$ divisor $E \subseteq
\mb{P}^2$, and with $\mc{L} = \mc{O}(1) \vert_E$.  In fact, we
restrict our attention here to what is in many ways the most
interesting regular algebra, the \emph{generic Sklyanin algebra} $S$
(see Section~\ref{sklyanin-sec} for a definition by presentation).
In this case, $g$ is central, $E$ is a nonsingular elliptic curve,
and the automorphism $\sigma: E \to E$ has infinite order.

The main objects of study in this paper will be certain subalgebras
of $S$ generated in degree $3$. It is convenient to work instead in
the $3$-Veronese $T = S^{(3)} = \bigoplus_{n \geq 0} S_{3n}$, and
construct degree-1-generated subalgebras of $T$.  Now $g \in T_1$
and $T/Tg \cong B(E, \mc{L}_3, \sigma^3)$.  Let $\overline{x}$
indicate the image of $x \in S$ under the quotient map $S \mapsto
S/Sg$. Given any effective Weil divisor $D$ on $E$ with $0 \leq \deg
D \leq 7$, we let $V(D) = \{x \in T_1 | \overline{x} \in \HB^0(E,
\mc{I}_D \otimes \mc{L}_3) \}$, where here $\mc{I}_D =
\mc{O}_E(-D)$.  In other words, $\overline{V(D)}$ consists exactly
of those global sections of $\mc{L}_3$ which vanish along the
divisor $D$.  We define $R(D) = k \langle V(D) \rangle \subseteq T$.
One goal of the paper is to prove that the rings $R(D)$ have many
nice properties (with the notable exception of finite global
dimension). We summarize these in the following theorem.  The
definitions of the properties involved are reviewed in
Section~\ref{background-sec}.
\begin{theorem}
\label{RD-props-thm} Let $R = R(D) \subseteq T$ as above, where $D$
is an effective divisor on $E$ of degree $d$, with $0 \leq d \leq
7$.
\begin{enumerate}
\item (Theorems~\ref{HS-thm}, \ref{Rop-thm}) $R/Rg \cong B(E, \mc{N}, \sigma^3)$, where $\mc{N} = \mc{I}_D
\otimes \mc{L}_3$ is a $\sigma^3$-ample invertible sheaf on $E$ of
degree $9 - d$.  The Hilbert series of $R$ is $\hs_R(t) = \D
\frac{t^2 + (7-d)t + 1}{(1-t)^3}$ and $Q_{\gra}(R) = Q_{\gra}(T)$.
$R$ has infinite global dimension.
\item (Theorems~\ref{basic-prop-thm}, \ref{max-order-thm})
$R$ is strongly noetherian, satisfies the Artin-Zhang $\chi$
conditions, is Auslander-Gorenstein and Cohen Macaulay.  The
noncommutative projective scheme $\rproj R$ has cohomological
dimension $2$.  $R$ is a maximal order.
\item (Theorem~\ref{min-ideal-thm}) There is a homogeneous ideal $I$ of $R$ with $\GK R/I \leq 1$
which is essentially minimal among such ideals in the following
sense:  given a homogeneous ideal $J$ of $R$ with $\GK R/J \leq 1$,
then $J \supseteq I_{\geq n}$ for some $n \geq 0$.
\end{enumerate}
\end{theorem}
\noindent  We make a comment on the significance of part (3) of the
theorem, which is actually the part on which we expend the most
effort in the paper.  The generic Sklyanin algebra $S$ and its
Veronese ring $T$ are known to have no homogeneous factor rings of
GK-dimension $1$. The rings $R(D)$ sometimes do, so the point of
part (3) above is that such factor rings can be strongly controlled.
This fact is needed in the proof of our main result, which can be
found in Section~\ref{main-thm-sec}.  It classifies degree
$1$-generated orders in $T$, as follows.
\begin{theorem}
\label{deg3-class-thm} Let $V \subseteq T_1$ and let $A = k \langle
V \rangle \subseteq T$. Assume that $Q_{\gra}(A) = Q_{\gra}(T)$.
\begin{enumerate}
\item There is a unique effective divisor $D$ on $E$ with $0 \leq
\deg D \leq 7$ such that $A \subseteq R(D)$ with $A$ and $R(D)$
equivalent orders. In particular, the rings $R(D)$ are exactly the
degree-$1$ generated subalgebras of $T$ which are maximal orders in
$Q_{\gra}(T)$.

\item
 If $A$ is noetherian, then in part (1) $A \subseteq R(D)$ is a finite ring
extension.  If $g \in A$, then $A$ is indeed noetherian.
\end{enumerate}
\end{theorem}
\noindent
When $g \not \in A$ in the preceding theorem, then it can happen
that $A$ is not noetherian and $A \subseteq R(D)$ is not a finite
ring extension; see Section~\ref{example-sec}.  Also, similar
methods can be used to analyze the subalgebras of $S$ generated in
degree $1$, which was in fact the original project we attempted.
Since $\dim_k S_1 = 3$, the only interesting degree-1-generated
subalgebras are $A = k \langle V \rangle$, where $V \subseteq S_1$
with $\dim_k V = 2$.  We show in Theorem~\ref{deg1-gen-thm} below
that for such an $A$, either the $3$-Veronese ring $A^{(3)}$ is
equal to the ring $R(D)$ for a divisor $D$ of degree $3$, or else
$A$ equals $S$ in all large degrees.

Next, we explain how our results relate to the project to classify
noncommutative projective surfaces, for which we need to review a
few more definitions.  See \cite{SV} for a survey of the field. If
$A$ is a cfg domain with $\GK(A) = d+1$, then informally we think of
$A$ as corresponding to a $d$-dimensional noncommutative projective
variety.  Let $\rGr A$ be the category of $\mb{Z}$-graded right
$A$-modules $M$, and let $\rTors A$ be the full subcategory of
graded modules $M$ such that for every $m \in M$, $mA_{\geq n} = 0$
for some $n$.  The quotient category $\rQgr A = \rGr A /\rTors A$ is
called the \emph{quasi-scheme} associated to $A$. The graded
quotient ring $Q_{\gra}(A)$ is isomorphic to a skew-Laurent ring
$D[t, t^{-1}; \sigma]$, for some division ring $D$ and automorphism
$\sigma: D \to D$.   We call $D$ the \emph{(skew) field of
functions} of $A$.  In general, the ring-theoretic version of the
problem of classification of noncommutative surfaces has two parts.
First, what are the $(D, \sigma)$ which occur in the graded quotient
rings of cfg domains $A$ of dimension $3$?  Second, for a given $Q =
D[t, t^{-1}; \sigma]$ where $D$ has transcendence degree $2$, can
one classify in some way the algebras $A$ with $Q_{\gra}(A) = Q$?

We say that $A$ is \emph{birationally commutative} if its skew field
of functions $D$ is actually a field.  The second question above has
a quite satisfactory answer in the birationally commutative case. If
$B = B(X, \mc{L}, \sigma)$ is a twisted homogeneous coordinate ring
with $X$ an integral surface and $\mc{L}$ $\sigma$-ample, then $Q =
Q_{\gra}(B) \cong k(X)[t, t^{-1}; \wt{\sigma}]$, where $\wt{\sigma}$
is the automorphism on the field of rational functions $k(X)$
induced by pullback by $\sigma$. Noetherian cfg algebras $A$ with
$Q_{\gra}(A) = Q_{\gra}(B)$ have now been classified  \cite{RS},
\cite{Si}.  In large degrees, such $A$ are either twisted
homogeneous coordinate rings over varieties $Y$ birational to $X$,
or one of a few kinds of closely related subrings of them
(na{\"\i}ve blowup rings, ADC rings, or idealizer rings in either of
these.) Moreover, geometrically, the quasi-scheme $\rQgr B$ of a
twisted homogeneous coordinate ring $B(Y, \mc{M}, \tau)$ is
equivalent to $\rQcoh Y$, the category of quasi-coherent sheaves on
$Y$.   The quasi-schemes of the other kinds of examples $A$ behave
as further blowups of the categories $\rQcoh Y$ in some generalized
noncommutative sense (``na{\"\i}ve blowups").

The main theorem of this paper, Theorem~\ref{deg3-class-thm}, is an
attempt to begin to classify some algebras which are not
birationally commutative. It is natural to start by considering
rings $A$ with $Q_{\gra}(A) = Q_{\gra}(S)$ for the generic Sklyanin
algebra $S$, since this is one of the most well-studied cfg algebras
of dimension 3.   Even in the commutative classification of
surfaces, one concentrates first on nonsingular surfaces, so it is
reasonable to focus our attention here on maximal orders: being a
maximal order is the noncommutative analog of being integrally
closed.

The intuition from the commutative case leads one to expect that the
quasi-schemes of the algebras in such a classification should be
related to $\rQgr S$ via some kind of blowup procedure. In fact, the
quasi-schemes $\rQgr R(D)$ are the same as certain iterated
noncommutative blowups of the Sklyanin $\mb{P}^2$, in the sense of
Van den Bergh's theory \cite{VdB}. Starting with the Sklyanin
algebra $S$ and a divisor $D$ of degree at most $8$ on $E$, Van den
Bergh's construction yields a blowup of the Sklyanin $\mb{P}^2$
along the divisor $D$ (see \cite[Section 11]{VdB}); this is a
quasi-scheme $\rQgr A(D)$ for some ring $A(D)$. Van den Bergh proves
that $\rQgr A(D)$ has many of the same formal properties as a
commutative blowup, when $D$ is in general position \cite[Theorem
11.1.3]{VdB}. Van den Bergh's work is primarily category-theoretic,
essentially defining a blowup by constructing a Rees ring over a
quasi-scheme, working in an appropriate category of functors.
Because the details are rather intricate, it is not obvious that the
ring $A(D)$ is isomorphic to our ring $R(D)$.  Only recently, after
most of the work in this paper was completed, did we fully verify
that this is indeed the case.  We postpone the demonstration of this
to a future paper, in which the connections between our results and
Van den Bergh's can be more fully explored.  In the meantime, we
thought would be most helpful to the reader to keep our presentation
here entirely independent of \cite{VdB}.  We should remark, however,
that once one knows that the rings appearing in \cite{VdB} are the
same, some of the most basic properties of the rings $R(D)$, in
particular the Hilbert series (Theorem~\ref{RD-props-thm}(1)),
follow in a quite different way from Van den Bergh's methods.

To close, we discuss some questions for future study.
\begin{question}
Can one classify all cfg maximal orders $A$ inside the generic
Sklyanin algebra $S$ with $Q = Q_{\gra}(A) = Q_{\gra}(S^{(n)})$ for
some $n$?  More generally, can we classify all cfg maximal orders in
$Q_{\gra}(S^{(n)})$?
\end{question}
\noindent So far, we have classified only those subalgebras $A$ of
$S$ generated in degree $1$ or $3$.  Based on the connection with
\cite{VdB}, there should also be subalgebras of $S$ corresponding to
blowups at $8$ points, which have presumably fallen outside the
scope of our results in this paper because they will not be
generated in any single degree. It seems possible that there are no
other subalgebras of $S$ which are maximal orders, except those that
have a Veronese ring in common with one of these blowups at $8$ or
fewer points. To answer the more general second question, one needs
to find first the ``minimal models" for this birational class. For
example, presumably the regular algebra with $2$ generators and $2$
cubic relations of Sklyanin type (the ``Sklyanin $\mb{P}^1 \times
\mb{P}^1$") is also contained in this quotient ring, and so one also
has it and its subalgebras as examples.

\begin{question}
If one replaces the generic Sklyanin algebra $S$ by other non-PI
AS-regular algebras of dimension $3$, can one prove theorems
analogous to the ones in this paper, or ultimately classify all cfg
algebras which are maximal orders with the same quotient ring?
\end{question}

\begin{question}
What is a presentation of the ring $R(D)$ by generators and
relations?
\end{question}
\noindent Answering this final question would be useful in order to
better understand the deformation theory of these blowup algebras.



\noindent \textbf{Acknowledgments.}  We thank Mike Artin, Michel Van
den Bergh, Colin Ingalls, and Paul Smith for helpful conversations.
In addition, a computer program written by Jay Gill at MIT as an
undergraduate research project was useful for intuition at the
beginning stages of this project.

\section{Background and twisted homogeneous coordinate rings}
\label{background-sec}

In this section, we review some definitions and basic background
material, including some facts about twisted homogeneous coordinate
rings $B(X, \mc{L}, \sigma)$.  Throughout this paper, $k$ will stand
for an algebraically closed field.  All algebras $A$ of interest
will be cfg $k$-algebras, as defined in the introduction. We make no
assumptions on the cardinality of $k$, but if the reader is happy to
assume that $k$ is uncountable, a few later arguments can be
streamlined slightly.  In general, the hypotheses and notation in
force in each particular section will be announced near the
beginning of that section.

For completeness, we review most of the definitions we will use in
the paper in this section.   Some of these we will need only
incidentally later, however, so the expert reader might wish to skip
this section and refer back to it when necessary.
We begin by reviewing the theory of noncommutative projective
schemes, which was developed in \cite{AZ1}.  For convenience, assume
that $A$ is a noetherian cfg $k$-algebra in the following
definitions.  Let $\rGr A$ be the category of $\mb{Z}$-graded right
$A$-modules. The general convention is to use a lowercase name to
indicate the subcategory of noetherian objects in a category.  For
example, $\rgr A \subseteq \rGr A$ is the full subcategory of
finitely generated $\mb{Z}$-graded
$A$-modules.  
Let $\rTors A \subseteq \rGr A$ be the subcategory of modules $M$
with the property that for every $m \in M$, $m A_{\geq n} = 0$ for
some $n \geq 1$.
As in the introduction, the quotient category $\rQgr A = \rGr
A/\rTors A$ is called the \emph{quasi-scheme} associated to $A$. Let
$\pi: \rGr A \to \rQgr A$ be the canonical quotient functor. A
review of the formalities of quotient categories can be found in
\cite{AZ1}.  In the case at hand, one may think of $\rqgr A$ as the
category of \emph{tails} of modules; in other words, two modules $M,
N \in \rgr A$ become isomorphic in $\rqgr A$ if and only if $M_{\geq
n} \cong N_{\geq n}$ for some $n$.
The \emph{noncommutative projective scheme} associated to $A$ is the
pair  $\rproj A = (\rqgr A, \pi(A_A))$; it remembers the
\emph{distinguished object} $\pi(A)$, which plays the role of the
structure sheaf.  For $M \in \rgr A$ and $n \in \mb{Z}$, $M[n] \in
\rgr A$ is the same module as $M$ but with the grading shifted so
that $M[n]_i = M_{i+n}$ for all $i \in \mb{Z}$.

Let $k = k_A$ be the right module $A/A_{\geq 1}$. The noetherian
algebra $A$ \emph{satisfies $\chi_i$ (on the right)} if $\dim_k
\ext^j_A(k, M) < \infty$ for all $M \in \rgr A$ and all $1 \leq j
\leq i$.  $A$ \emph{satisfies $\chi$ (on the right)} if it satisfies
$\chi_i$ for all $i \geq 1$.  Of course, one defines the left
$\chi$-conditions analogously.  The \emph{cohomology} groups of an
object $\mc{M} \in \rqgr A$ are defined by $\HB^i(\mc{M}) =
\ext_{\rqgr A}^i(\pi(A), \mc{M})$ for all $i \geq 0$.
The \emph{cohomological dimension} of $\rproj A$ is $\cd(\rproj A) =
\min \{i | \HB^i(\mc{M}) \neq 0\ \text{for some}\ \mc{M} \in \rqgr A
\}$.

Next, we review definitions related to Hilbert series and growth of
modules.  Let $A$ be any cfg $k$-algebra.  Any $M \in \rgr A$ has a
\emph{Hilbert function} $f_M(n): \mb{Z} \to \mb{N}$, defined by
$f_M(n) = \dim_k M_n$; the \emph{Hilbert series} of $M$ is the
corresponding formal Laurent power series $\hs_M(t) = \sum_{n \in
\mb{Z}} (\dim_k M_n)t^n$. We sometimes use the following partial
order on Hilbert series: $\sum a_n t^n \leq \sum b_n t^n$ if $a_n
\leq b_n$ for all $n$.  In this graded setting, we can define the
\emph{GK-dimension} of a nonzero module $M \in \rgr A$ as $\GK(M) =
[\limsup_{n \to \infty} \log_n f_M(n)] + 1$.
The module $M \in \rgr A$ has a \emph{Hilbert polynomial} if there
is a polynomial $q(x) \in \mb{Q}[x]$ such that $q(n) = f_M(n)$ for
all $n \gg 0$; in this case, writing $f(x) = a_mx^m + \dots + a_0$
with $a_m \neq 0$, then $\GK M = m + 1$ and $m! a_m$ is an integer,
the \emph{multiplicity} of $M$.  For any $d \geq 1$, we have the
$d$th Veronese ring of $A$, $A^{(d)} = \bigoplus_{n \geq 0} A_{nd}$,
which is graded by the index $n$, unless stated otherwise. A module
$M \in \rgr A$ has a \emph{Hilbert quasi-polynomial of period $d$}
for some $d \geq 1$, if $\bigoplus_{n \geq 0} M_{i + nd}$ has a
Hilbert polynomial as an
 $A^{(d)}$-module for each $0 \leq i \leq d-1$.
We say that $A$ is \emph{generated in degree $1$} if it is generated
as a $k$-algebra by $A_1$.  Given a cfg $k$-algebra $A$ which is
generated in degree $1$, we say that $M \in \rgr A$ is a \emph{point
module} if $M$ is cyclic and $\hs_M(t) = 1/(1-t)$.  A module $M \in
\rgr A$ with  $\GK(M) = d$ is
\emph{$d$-critical} if $\GK(M/N) <
d$ for all nonzero submodules $N \subseteq M$.

The $k$-algebra $A$ is \emph{strongly noetherian} if $A \otimes_k C$
is noetherian for all commutative noetherian $k$-algebras $C$.
As usual, a ring extension $A \subseteq B$ is \emph{finite} if $B_A$
and $_A B$ are finitely generated $A$-modules.  We will often use
below the graded Nakayama lemma, which states that $M \in \rGr A$ is
finitely generated if and only if $\dim_k M/MA_{\geq 1} < \infty$
(and, in fact, $X \subseteq M$ is a generating subset if and only if
$X$ spans $M/MA_{\geq 1}$ as a $k$-vector space.)


The last definitions we review are some important homological
properties which involve ungraded modules. For a right $A$-module
$M$, define $j(M_A) = \inf \{i | \ext^i(M,A) \neq 0 \} \in \mb{N}
\cup \{ \infty \}$, which is called the \emph{grade} of $M$.  The
grade of a left module is defined analogously.
A right module $M$ \emph{satisfies the Auslander condition} if for
all $i \geq 0$ and all left submodules $N \subseteq \ext^i(M,A)$,
one has $j( { _A} N) \geq i$; the definition of Auslander for a left
module is symmetric. The ring $A$ is \emph{Auslander-Gorenstein} if
the modules $A_A$ and $_A A$ have the same finite injective
dimension $d$, and every finitely generated left and right
$A$-module satisfies the Auslander condition. Suppose in addition
that $\GK(A)$ is an integer.  Then $A$ is \emph{Cohen-Macaulay} if
$\GK(M) + j(M) = \GK(A)$ for all finitely generated left and right
modules $M$.

Now we move on to a review of twisted homogeneous coordinate rings.
Recall that in the introduction, we defined the twisted homogeneous
coordinate ring $B(X, \mc{L}, \sigma) = \bigoplus_{n \geq 0}
\HB^0(X, \mc{L}_n)$ for a given projective $k$-scheme $X$,
invertible sheaf $\mc{L}$, and automorphism $\sigma: X \to X$.  We
now give a few more details about this construction; for more
information, see \cite{AV} and \cite{Ke2}.  For any coherent sheaf
$\mc{F}$ on $X$ we adopt the notation $\mc{F}^{\sigma}$ for the
pullback $\sigma^*\mc{F}$, so in this notation we have $\mc{L}_n =
\mc{L} \otimes \mc{L}^{\sigma} \otimes \dots \mc{L}^{\sigma^{n-1}}$
for $n \geq 1$.   The sheaf $\mc{L}$ is called \emph{$\sigma$-ample}
if for any coherent sheaf $\mc{F}$ on $X$, $\HB^i(X, \mc{F} \otimes
\mc{L}_n) = 0$ for all $i \geq 1$ and $n \gg 0$.  We always assume
that $\mc{L}$ is $\sigma$-ample when constructing a twisted
homogeneous coordinate ring.

Many of the definitions discussed earlier play out nicely when $A =
B =  B(X, \mc{L}, \sigma)$ is a twisted homogeneous coordinate ring,
where $\mc{L}$ is $\sigma$-ample.  If $X$ is integral (as it will be
in all of the applications in this paper), the graded quotient ring
of $B$ may be explicitly described as $Q_{\rm gr}(B) \cong  k(X)[t,
t^{-1}; \sigma]$, where $\sigma: k(X) \to k(X)$ is the map induced
by pullback of rational functions by $\sigma: X \to X$.
In this case, as mentioned already in the introduction, the
quasi-scheme $\rQgr B$ is equivalent as a category to  $\rQcoh X$,
the category of quasi-coherent sheaves on $X$ \cite[Theorem
1.3]{AV}.  Restricting to noetherian objects for convenience, the
equivalence is given explicitly by the functor $G: \coh X \to \rqgr
B$ with formula
\begin{equation}
\label{cat-equiv-eq}
  G(\mc{F}) = \pi(\bigoplus_{n \geq 0} \HB^0(X, \mc{F} \otimes
\mc{L}_n)).
\end{equation}
For later reference, we record here some other well-known properties
of twisted homogeneous coordinate rings which follow quickly from
the literature.
\begin{lemma} \label{B-prop}  Let $X$ be a projective scheme,
$\sigma: X \to X$ an automorphism, and $\mc{L}$ a $\sigma$-ample
invertible sheaf on $X$.  Let $B = B(X, \mc{L}, \sigma)$.
\begin{enumerate}

\item $B$ is strongly noetherian.

\item $B$ satisfies $\chi$ on the left and right, and $\on{cd}(\rproj
B) = \dim X$.

\item If $X$ is a integral nonsingular curve, then every module $M
\in \rgr B$ has a Hilbert polynomial of the form $\dim_k M_n = an +
b$ for $n \gg 0$, some $a, b \in \mb{Z}$ with $a \geq 0$.

\item if $X = E$ is a nonsingular elliptic curve, then $B$ is
Auslander-Gorenstein of dimension $2$ and Cohen-Macaulay.
\end{enumerate}
\end{lemma}
\begin{proof}
(1) This is \cite[Proposition 4.13]{ASZ}.

(2) These results are implicit in Artin and Zhang's original paper
\cite{AZ1} but are not stated in this way.  The fact that $B$
satisfies $\chi$ follows from \cite[Theorem 4.5, Corollary 7.5, and
p. 262]{AZ1} and the fact that $\mc{L}$ is $\sigma$-ample.  For a
working out of the details, see Keeler's thesis \cite[Proposition
2.5.3, Proposition 2.5.8]{Ke1}.

Since the category equivalence \eqref{cat-equiv-eq} takes $\mc{O}_X$
to $\pi(B)$, it also shows that $\cd(\rproj B)$ is the same as the
cohomological dimension of the scheme $X$, i.e. $\max \{i | \HB^i(X,
\mc{F}) \neq 0\ \text{for some}\ \mc{F} \in \coh X \}$, which it is
well-known is equal to $\dim X$ for a projective scheme $X$.

(3) This follows quickly from \eqref{cat-equiv-eq}
and the Riemann-Roch theorem.

%

(4) This is \cite[Theorem 6.6]{Lev}, except that Levasseur assumes
the restriction $\deg \mc{L} \geq 3$, so that $\mc{L}$ is very
ample. But this restriction is only used in the proof of a different
part of the theorem showing that every graded $B$-module has a
Hilbert polynomial (and the restriction is also unnecessary to prove
that, as we saw in (3)).
\end{proof}

\section{Twisted homogeneous coordinate rings over an
elliptic curve} \label{elliptic-sec}

Theorem~\ref{RD-props-thm} from the introduction shows that the
rings $R$ of main interest in this paper will all have a central
element $g$ such that $R/Rg \cong B(E, \mc{L}, \sigma)$, where $E$
is a nonsingular elliptic curve and $\sigma$ has infinite order.
Thus these special twisted homogeneous coordinate rings will play an
important role below, and in this section we prove some results that
are particular to this special case.

For the rest of this section we restrict to the following setup. Let
$E$ be a nonsingular elliptic curve with fixed basepoint $p_0$ for
the group structure.  Let $\sigma: E \to E$ be an automorphism given
by translation $x \mapsto x + r$ in the group structure, for some
point $r$ of infinite order in the group.  It is standard that any
automorphism of infinite order on an elliptic curve is such a
translation, so we are really working with any infinite order
automorphism of $E$.  The notation of this paragraph is fixed for
the rest of this section.

We will freely use basic properties of nonsingular curves and
divisors on them which can be found in \cite[Chapter IV]{Ha}, for
example the Riemann-Roch theorem. Another basic result we use
frequently is Abel's theorem: two Weil divisors $D, D'$ on $E$ are
linearly equivalent ($D \sim D'$) if and only if $\deg D = \deg D'$
and $D$ and $D'$ give the same result when summed in the group
structure of $E$.  Note that the symbol $+$ is used both for the
group structure of $E$ and for addition in the group $\Div E$ of
Weil divisors of $E$; the meaning will be clear from context.

The property of $\sigma$-ampleness is easy to understand on curves;
in particular, an invertible sheaf $\mc{L}$ on $E$ is $\sigma$-ample
if and only if $\mc{L}$ is ample, in other words if and only if
$\mc{L}$ has positive degree \cite[Theorem 1.3]{Ke2}.  Our first
goal is to show that $B(E, \mc{L}, \sigma)$ is generated in degree
$1$ as long as $\mc{L}$ does not have very small degree.  This is
similar to the results in \cite[Section 7]{ATV1};  in fact, some
parts of the next lemma also follow from \cite[Proposition 7.17,
Lemma 7.21]{ATV1}, but we give a full proof for the reader's
convenience.
\begin{lemma}
\label{inv-section-lem} \label{inv-section-prop}
\label{gen-deg1-lem}
\begin{enumerate}
\item
Let $\mc{L}$ and $\mc{M}$ be invertible sheaves on $E$ with $\deg
\mc{L} \geq 2$ and $\deg \mc{M} \geq 2$.  Consider the natural map
$\phi: \HB^0(E, \mc{L}) \otimes \HB^0(E, \mc{M}) \to \HB^0(E, \mc{L}
\otimes \mc{M})$.  The map $\phi$ is surjective unless $\deg \mc{L}
= \deg \mc{M} = 2$ and $\mc{L} \cong \mc{M}$, in which case $\dim_k
\im \phi  = 3$.
\item  Let $\mc{L}$ be an invertible sheaf on $E$ with $\deg \mc{L} \geq 2$.  Then
$B = B(E, \mc{L}, \sigma)$ is generated in degree $1$.
\end{enumerate}

\end{lemma}
\begin{proof}
If $\mc{L} \cong \mc{M}$ with $\deg \mc{L} = 2$, then the natural
map is clearly not surjective.  For, in that case we might as well
assume that $\mc{M} = \mc{L}$, and $\dim_k \HB^0(E, \mc{L}) = 2$ and
$\dim_k \HB^0(E, \mc{L}^{\otimes 2}) = 4$ by Riemann-Roch.  Then
writing $\HB^0(E, \mc{L}) = kx + ky$, we have $\phi(x \otimes y) =
\phi(y \otimes x)$, and so $\dim_k \im \phi \leq 3$.   It is easy to
check that in fact we have equality here, as asserted.

If $\mc{L} \cong \mc{M}$ where $\deg \mc{L} \geq 3$, again we may
assume that $\mc{M} = \mc{L}$.  Then the surjectivity of $\phi$ is a
special case of the standard result that any invertible sheaf of
degree at least $2g + 1$ on a curve of genus $g$ is normally
generated. See, for example, \cite[Definition 1.8.50 and Theorem
1.8.53]{Lz1}.

Finally, assume that $\mc{M} \not \cong \mc{L}$, where by switching
the sheaves if necessary we can assume that $\deg \mc{L} \leq \deg
\mc{M}$.  In this case, we show a slightly stronger result. It is
well-known that any globally generated invertible sheaf $\mc{L}$ on
$E$ is generated by two sections.  We claim that choosing any two
sections $x, y \in \HB^0(E, \mc{L})$ which generate the sheaf
$\mc{L}$, then putting $W = kx + ky$ we even have $\phi(W \otimes
\HB^0(E, \mc{M})) = \HB^0(E, \mc{L} \otimes \mc{M})$. Considering
the exact sequence
\[
0 \to \mc{K} \to W \otimes \mc{O}_E \to \mc{L} \to 0,
\]
it is clear that $\mc{K}$ is also invertible, and in fact it follows
from a consideration of determinants (\cite[Exercise
II.5.16(d)]{Ha}) that $\mc{K} \cong \mc{L}^{-1}$.  Tensoring with
$\mc{M}$ and taking the cohomology long exact sequence, we see that
to prove the claim it suffices to prove that $\HB^1(E, \mc{L}^{-1}
\otimes \mc{M}) = 0$.  But since $\mc{N} = \mc{L}^{-1} \otimes
\mc{M}$ has $\mc{N} \not \cong \mc{O}_E$ and $\deg \mc{N} \geq 0$,
it is standard that $\HB^1(E, \mc{N}) = 0$  \cite[Section IV.1]{Ha}.

(2) It suffices to show that the map
\[
\HB^0(E, \mc{L}) \otimes \HB^0(E, \mc{L}_{n}^{\sigma}) \to \HB^0(E,
\mc{L}_{n+1})
\]
is surjective for all $n \geq 1$.  This follows immediately from
part (1), unless $\deg \mc{L} = 2$, $n = 1$, and $\mc{L}^{\sigma}
\cong \mc{L}$.  But since $\sigma$ is translation by a point of
infinite order on $E$, $\sigma$ does not fix the linear equivalence
class of any nonzero effective divisor; thus $\mc{L}^{\sigma} \not
\cong \mc{L}$.
\end{proof}
\noindent  In fact, it is not hard to show that $B(E, \mc{L},
\sigma)$ is generated as an algebra by $B_1$ and $B_2$ if $\deg
\mc{L} = 1$, but we will not need this result.

Next, we study subalgebras of $B(E, \mc{L}, \sigma)$.  The structure
of these is strongly constrained by the results of Artin and
Stafford on cfg domains of GK-dimension $2$.  Recall that (following
\cite{RRZ}) a cfg algebra $A$ is called \emph{projectively simple}
if every nonzero homogeneous ideal $I$ of $A$ satisfies $\dim_k A/I
< \infty$.  Using the Artin-Stafford theory, we now prove the
following.
\begin{lemma}
\label{B-subring-lem} Consider $B = B(E, \mc{L}, \sigma)$ where
$\deg \mc{L} \geq 1$.  Let $A$ be a cfg subalgebra of $B$ with
$\dim_k A_n \geq 2$ for some $n \geq 1$.
\begin{enumerate}
\item  $\GK A = 2$, $A$ is noetherian, and $A$ is projectively simple.

\item Suppose that $A_n \neq 0$ for all $n \gg 0$.  Then there is another nonsingular elliptic curve $F$ with
automorphism $\tau$, such that $A$ is equal in all large degrees to
a ring of the form
\begin{equation}
\label{GK2-classify-eq} \bigoplus_{n \geq 0} \HB^0(F, \mc{I} \otimes
\mc{M} \otimes \mc{M}^{\tau} \otimes \mc{M}^{\tau^{(n-1)}}),
\end{equation}
where $\mc{M}$ is an invertible sheaf on $F$ of positive degree and
$\mc{I}$ is an ideal sheaf.    Moreover,
there is a finite surjective morphism $\theta: E \to F$ such that
$\tau \theta = \theta \sigma$, and $\tau$ is again translation by a
point of infinite order if we take basepoint $\theta(p_0)$ for the
group structure of $F$.

\item Suppose that $A$ is generated in degree $1$.
Let $\mc{N}$ be the sheaf generated on $E$ by the sections in $A_1
\subseteq H^0(E, \mc{L})$.  Then $\mc{I} = \mc{O}_E$ in
\eqref{GK2-classify-eq}, and so $A$ is equal in large degree to $A'
= B(F, \mc{M}, \tau)$.  We also have $\deg_F \mc{M} \geq 2$; $\mc{N}
= \theta^*\mc{M}$; and $\deg \mc{N} = (\deg \theta)(\deg \mc{M})$.
Finally, $A \subseteq B(E, \mc{N}, \sigma)$ is a finite ring
extension.
\end{enumerate}
\end{lemma}
\begin{proof}
(1) It is standard that since $k$ is algebraically closed, a cfg
domain $A$ with $\GK A \leq 1$ is just a subalgebra of $k[t]$ where
$t$ has positive degree.  (To see this, note that $\GK A \leq 1$
implies that $\dim_k A_n$ is uniformly bounded, and so $Q_{\gra}(A)
\cong D[t, t^{-1}; \rho]$ where $D$ is a finite dimensional division
algebra over $k$.)  The hypothesis on $A$ thus implies that $\GK A =
2$, using Bergman's gap theorem.

Note that since $A$ is a domain, $\mc{S} = \{n | A_n \neq 0 \}$ must
be a sub-semigroup of $\mb{N}$, so it is easy to see that $\mc{S}$
contains all large multiples of $d = \gcd(\mc{S})$. In particular,
replacing $A$ by its Veronese ring $A^{(d)}$ just amounts to a
regrading, and the hypotheses are preserved (since $B^{(d)} \cong
B(E, \mc{L}_d, \sigma^d)$).  Clearly the desired conclusions that
$A$ is noetherian and projectively simple are unaffected by this
change in grading.  Thus we assume from now on that $A_n \neq 0$ for
all $n \gg 0$, and we prove that $A$ is noetherian and projectively
simple in part (2).

(2) Since $A$ is a domain of finite GK-dimension, its graded
quotient ring exists and has the form
\[
Q_{\rm gr}(A) = L[t, t^{-1}; \sigma \vert_L] \subseteq Q_{\rm gr}(B)
\cong k(E)[t, t^{-1}; \sigma],
\]
where we can take the same $t$ of degree $1$ in both quotient rings
since $A_n \neq 0$ for all $n \gg 0$. Here, $\sigma: k(E) \to k(E)$
is the automorphism on rational functions induced by $\sigma: E \to
E$, and necessarily $L$ is a $\sigma$-fixed subfield of $k(E)$.
Since $\GK A = 2$,
$\trdeg L/k = 1$ \cite[Theorem 0.1]{AS}, and so $L \subseteq k(E)$
is a finite degree extension. Since $\sigma: k(E) \to k(E)$ has
infinite order, $\sigma \vert_L$ has infinite order.

Now the Artin-Stafford theory shows that $A$ determines a projective
curve $F$ with $k(F) = L$, and an automorphism $\tau: F \to F$
inducing $\sigma \vert_L$ on rational functions \cite[Proposition
3.3]{AS}. Since $E$ is nonsingular, the inclusion of fields $k(F)
\subseteq k(E)$ induces a surjective finite morphism of curves
$\theta: E \to F$, such that $\tau \theta = \theta \sigma$.  If $F$
has a singular locus $S$, then $S$ is $\tau$-invariant, and so
$\theta^{-1}(S)$ is a $\sigma$-invariant set of points in $E$. Since
$\sigma$ is a translation by a point of infinite order, it has no
nonempty $\sigma$-invariant finite subsets, and so $\theta^{-1}(S)$
and hence also $S$ is empty.  Thus $F$ is also nonsingular.  By
Hurwitz's Theorem \cite[Corollary IV.2.4]{Ha}, $F$ has genus $1$ or
$0$.  If $F$ has genus $0$, then $F \cong \mb{P}^1$ and $\tau$
necessarily has a fixed point $p$.  Then $\theta^{-1}(p)$ is a
finite non-empty $\sigma$-invariant subset of $E$, again a
contradiction.  Thus $F$ has genus $1$, in other words $F$ is also a
nonsingular elliptic curve.  Recalling that $p_0$ is the fixed
basepoint for the group structure on $E$, choosing the basepoint $
\theta(p_0)$ for the group structure on $F$, then $\theta$ is a
homomorphism of groups \cite[Lemma 4.9]{Ha}.  Since $\sigma$ is the
translation $x \mapsto x + r$, then $\tau: F \to F$ is the
translation $y \mapsto y + \theta(r)$, where clearly $\theta(r)$ is
also a point of infinite order in the group of $F$.

Now \cite[Proposition 6.4]{AS} shows that since $\tau$ has no fixed
points, $A$ must be isomorphic in large degree to the ring in
\eqref{GK2-classify-eq}
for some invertible sheaf $\mc{M}$ of positive degree and some ideal
sheaf $\mc{I}$ on $F$.  It also follows that $A$ is noetherian
\cite[Theorem 5.6]{AS}, and projectively simple \cite[Theorem
5.11(2)]{AS}.

(3) Now assume that $A$ is generated in degree $1$, where $\mc{N}
\subseteq \mc{L}$ is the invertible sheaf generated on $E$ by the
sections in $A_1$.
Part (2) applies to $A$, and so $A$ has the form given in
\eqref{GK2-classify-eq} in large degree, necessarily with $\mc{I} =
\mc{O}_F$ by \cite[Theorem 4.7]{AS}. In other words, $A$ is
isomorphic in large degree to the twisted homogeneous coordinate
ring $A' = B(F, \mc{M}, \tau)$. Obviously if $\dim_k A_1 = 1$ then
$A \cong k[t]$, contradicting the hypothesis, so $\dim_k A_1 \geq
2$. Then $\deg \mc{M} \geq 2$. Clearly also the sections in $A_1
\subseteq \HB^0(F, \mc{M})$ must generate $\mc{M}$ on $F$, or else
$A$ cannot equal $B(F, \mc{M}, \tau)$ in large degree.  We must then
have $\theta^*(\mc{M}) = \mc{N}$, and $\deg \mc{N} = (\deg
\theta)(\deg \mc{M})$, where $\deg \theta$ is the number of points
in every fiber of $\theta$.  To check that $A = B(F, \mc{M}, \tau)
\subseteq B(E, \mc{N}, \sigma)$ is a finite extension of rings, we
prove that $_A B$ is finitely generated and leave the similar
verification that $B_A$ is finitely generated to the reader. It is
enough, by the graded Nakayama lemma, to show that $\dim_k B/A_{\geq
1}B < \infty$. The right ideal $A_{\geq 1}B$ of $B$ is, by the
equivalence of categories \eqref{cat-equiv-eq}, equal in large
degree to $\bigoplus_{n \geq 0} \HB^0(E, \mc{J} \otimes \mc{N}_n)$
for some ideal sheaf $\mc{J}$ on $E$.  Since $A_n$ already generates
$\mc{N}_n$ on $E$ for each $n$, necessarily $\mc{J} = \mc{O}_E$.
This shows that $\dim_k B/A_{\geq 1}B < \infty$, as required.
\end{proof}

Consider $B(E, \mc{N}, \sigma)$ where $\deg \mc{N} \geq 1$, so
$\mc{N}$ is $\sigma$-ample.  Given a module $N \in \rgr B$ with $\GK
N \leq 1$, it will be useful later to associate a Weil divisor to
$N$ which tracks the composition factors of the corresponding
coherent sheaf on $E$, in the following way.
\begin{definition}
\label{C-def} For any $N \in \rgr B(E, \mc{N}, \sigma)$ with $\GK N
\leq 1$, then $\pi(N) \in \rqgr B$ corresponds under the equivalence
of categories \eqref{cat-equiv-eq} to a torsion sheaf $\mc{F} \in
\coh E$. Let $C(N) = \sum_{p \in E}
[\operatorname{length}_{\mc{O}_{E, p}} \mc{F}_p] \cdot p \in \Div
E$.
\end{definition}
We want to make a few comments about point modules for $B =  B(E,
\mc{N}, \sigma)$.  We only want to discuss point modules for rings
which are generated in degree $1$, so because of
Lemma~\ref{inv-section-lem} we assume that $\deg \mc{N} \geq 2$.  In
particular, in this case $\mc{N}$ is generated by its global
sections \cite[Corollary IV.3.2]{Ha}, and it easily follows that for
each $q \in E$,  $P = P(q) = \bigoplus_{n \geq 0} \HB^0(E, k(q)
\otimes \mc{N}_n)$ is a point module for $B$, where here $k(q)$ is
the skyscraper sheaf at $q$.

In the last result of this section, we will study how the divisors
defined in Definition~\ref{C-def} are related for modules over two
different twisted homogeneous coordinate rings.  Let $\mc{N}
\subseteq \mc{N}'$ be invertible sheaves on $E$, with $2 \leq \deg
\mc{N} = \deg \mc{N}' -1$. Thus $\mc{N} = \mc{I}_p \otimes \mc{N'}$
for some unique point $p \in E$, where $\mc{I}_p = \mc{O}_E(-p)$ is
the ideal sheaf of the point $p$. Then we have $B = B(E, \mc{N},
\sigma) \subseteq B' = B(E, \mc{N'}, \sigma)$.
Given $q \in E$, we have the corresponding $B'$-point module
$P(q)_{B'} = \bigoplus_{n \geq 0} \HB^0(E, k(q) \otimes \mc{N}'_n)$
and the corresponding $B$-point module $Q(q)_{B} = \bigoplus_{n \geq
0} \HB^0(E, k(q) \otimes \mc{N}_n)$.
\begin{lemma}
\label{restrict-lem} Keep the notation of the previous paragraph.
\begin{enumerate}
\item $P(q)_B \in \rgr B$ and $\pi(P(q)_B) = \pi(Q(q))$ in $\rqgr B$.
\item For any $N \in \rgr B'$ with $\GK N \leq 1$, then $N_B \in \rgr B$ and $C(N_{B'}) =
C(N_B)$ in Definition~\ref{C-def}.
\end{enumerate}
\end{lemma}
\begin{proof}
(1) Consider
\[
P = P(q)_{B'} = \bigoplus_{n \geq 0} \HB^0(E, (\mc{O}_E/\mc{I}_q)
\otimes \mc{N}'_n).
\]
Given $n \geq 0$, we calculate $\{x \in B'_1 | P_n x = 0 \} =
\HB^0(E, \mc{I}_{\tau^n(q)} \otimes \mc{N}')$, and so
\[
\{ x \in B_1 | P_n x = 0 \} = \HB^0(E, \mc{I}_{\tau^n(q)} \otimes
\mc{N}') \cap B_1 = \begin{cases} \HB^0(E,
\mc{I}_{\tau^n(q)} \otimes \mc{N}) \subsetneq B_1 & \text{if}\ p \neq \tau^n(q) \\
B_1 & \text{if}\ p = \tau^n(q) \end{cases}
\]
In particular, $P_n B_1 = P_{n+1}$ if and only if $p \neq
\tau^n(q)$.   Since all orbits of $\tau$ are infinite, there is at
most one $n$ such that $p = \tau^n(q)$.  Thus for $n$ large enough,
say $n \geq n_0$, we see that $(P_{\geq n})_B$ is a shifted point
module for $B$.  In particular, $P_B$ is finitely generated. A
similar calculation to the one above also easily shows for $n \geq
n_0$ that $Q(q) = \bigoplus_{n \geq 0} \HB^0(E, k(q) \otimes
\mc{N}_n)$ satisfies $Q_{\geq n} \cong (P_{\geq n})_B$, since
$\on{r.ann}_B Q_n =\on{r.ann}_B P_n = \bigoplus_{m \geq 0} \HB^0(E,
\mc{I}_{\tau^n(q)} \otimes \mc{N}_m)$. It follows that $\pi(Q_B) =
\pi(P_B)$ in $\rqgr B$.

(2)  Suppose that $N$ is a $1$-critical $B'$-module.  Applying the
category equivalence \eqref{cat-equiv-eq}, clearly $G(N)$ is a
simple object of $\coh E$ and so is equal to a skyscraper sheaf
$k(q)$ for some $q$.  Then $N_{\geq n} \cong P(q)_{\geq n}$ for some
$n$.  Now given any $M \in \rgr B'$ with $\GK M = 1$, using the fact
that $M$ has constant Hilbert function by Lemma~\ref{B-prop}(3), it
is easy to prove that $M$ has a filtration $0 = M_0 \leq M_1 \leq
\dots \leq M_m = M$, in which the factors $M_i/M_{i-1}$ are either
finite-dimensional or $1$-critical (and so equal in large degree to
tails of point modules $P(q)_{\geq n}$.) Clearly $C(N)$ just
enumerates the points $q$ corresponding to the point module tails in
this filtration. Since the same interpretation holds for GK-1
modules over the ring $B$, both claims are now immediate from part
(1).
\end{proof}

\section{The generic Sklyanin algebra}
\label{sklyanin-sec} In this section, we review some basic
properties of the Sklyanin algebra $S$.  We then prove a
combinatorial proposition, which will be the key to the calculation
of the Hilbert series of the rings $R(D)$ of main interest.

For any $a,b,c \in k$, the Sklyanin algebra $S = S(a,b,c)$ is the
 $k$-algebra with presentation
\[
S(a,b,c) = k \langle x, y, z \rangle/\{axy + byx + cz^2, ayz + bzy +
cx^2, azx + bxz + cy^2 \}.
\]
The algebra $S$ is $\mb{N}$-graded where $x, y, z$ all have degree
$1$, and so $S$ is generated as a $k$-algebra by $S_1$.  It is
well-known that for very general choices of the parameters $a,b$,
and $c$, $S$ has the following properties: it is Artin-Schelter
regular with Hilbert series $h_S(t) = 1/(1-t)^3$; it has a unique up
to scalar central element $g \in S_3$; $S/Sg \cong B(E, \mc{L},
\sigma)$, where $E$ is a nonsingular elliptic curve embedded as a
degree $3$ divisor in $\mb{P}^2$ and $\mc{L} = \mc{O}(1) \vert_E$;
and $\sigma: E \to E$ has infinite order.  See \cite{ATV1, ATV2} for
more details, including the definition of regular algebra.  When $S$
has all of the above properties, we say that $S$ is a \emph{generic}
Sklyanin algebra. We indicate the image of any subset $V$ of $S$
under the quotient map $S \to S/Sg$ by $\overline{V}$.  Since
$\sigma$ has infinite order, fixing some base point $p_0 \in E$ for
the group law, then $\sigma$ is a translation $x \mapsto x + r$,
where $r$ has infinite order in the group.  For the rest of the
paper, $S$ stands for a generic Sklyanin algebra and the above
notation is also fixed.


It is known that every point module for $S$ is annihilated by $g$,
so the point modules for $S$ are the same as the point modules for
$S/Sg = B = B(E, \mc{L}, \sigma)$ \cite[Proposition 7.7{ii}]{ATV2}.
As we saw in the previous section, given a point $q \in E$ we have a
corresponding point module $B/J$ for $B$, where $J = J(q) =
\bigoplus_{n \geq 0} \HB^0(E, \mc{I}_q \otimes \mc{L}_n)$.  As an
$S$-module, this is the point module $S/I$, where $I = I(q) = \{x
\in S | \overline{x} \in J(q) \}$.  Since $S$ has the same Hilbert
series as a commutative polynomial ring in three variables, we have
$\dim_k S_1 = 3$ and so $\dim_k I(q)_1 = 2$; we set $W(q) = I(q)_1$
and we call any such $W(q) \subseteq S_1$ a \emph{point space}.  We
prove next some very useful formulas about products of point spaces
and copies of $S_1$.  Part (2) of the following lemma is also
derived by Ajitabh \cite[Lemma 2.3, Proposition 2.14(iv)]{Aj1}, but
we give the proof since it follows easily from
Lemma~\ref{inv-section-lem}.
\begin{lemma}
\label{ps-basic-lem}  Let $p,q \in E$.
\begin{enumerate}
\item $S_1W(\sigma(q)) = W(q)S_1$.

\item We have $\dim_k W(p) W(q) = 4$ if $q \neq \sigma^{-2}(p)$. On the other hand, $\dim_k W(p) W(\sigma^{-2}(p)) = 3$.
\end{enumerate}
\end{lemma}
\begin{proof}
(1) Since $\overline{S}_2$ and $S_2$ may be identified, we just note
that in $\overline{S}$, the product $\overline{S}_1
\overline{W(\sigma(q))}$ amounts to the image of the multiplication
map
\[
\HB^0(E, \mc{L}) \otimes \HB^0(E, (\mc{I}_{\sigma(q)} \otimes
\mc{L})^{\sigma}) \to \HB^0(E, \mc{L}_2),
\]
which is $\HB^0(E, \mc{I}_q \otimes \mc{L}_2)$ by
Lemma~\ref{inv-section-lem}(1).  In other words, the image is the
space of all sections of $\mc{L}_2$ vanishing along the point $q$.
The space $\overline{W(q)} \overline{S}_1$ is equal to the same
thing by a similar calculation.

(2) Again we identify $S_2$ and $\overline{S}_2$, and we are looking
for the image of the multiplication map
\[
\theta: \HB^0(E, \mc{I}_p \otimes \mc{L}) \otimes \HB^0(E, (\mc{I}_q
\otimes \mc{L})^{\sigma}) \to \HB^0(E, \mc{L}_2).
\]
As we saw in Lemma~\ref{inv-section-lem}(1), $\im(\theta)$ has
dimension $3$ if $\mc{I}_p \otimes \mc{L} \cong
\mc{I}_{\sigma^{-1}(q)} \otimes \mc{L}^{\sigma}$; otherwise, we have
$\im(\theta) = \HB^0(E, \mc{I}_{p} \otimes \mc{I}_{\sigma^{-1}(q)}
\otimes \mc{L}_2)$, and so $\dim_k \im(\theta) = 4$. Letting $\mc{L}
= \mc{O}_E(D)$ for some divisor $D$, we see that the case $\dim_k
\im(\theta) = 3$ occurs if and only if $D - p \sim \sigma^{-1}(D) -
\sigma^{-1}(q)$.  Since $\sigma$ is a translation $x \to x + r$,
using Abel's theorem this condition is equivalent to $p = q + 2r$ in
the group law, or $\sigma^2(q) = p$.
\end{proof}
Part (2) of the preceding lemma is also related to the following
explicit free resolution of a point module.
\begin{lemma}
\label{pt-syz-lem} Let $W = W(q) \subseteq S_1$ be a point space
with $W = kw + kx$.  By Lemma~\ref{ps-basic-lem}(2), there are
elements $y,z \in S_1$ such that $wy + xz = 0$, and with $ky + kz =
W(\sigma^{-2}(q))$.  Consider the complex
\[
0 \to S[-2] \overset{\begin{pmatrix} y \\ z
\end{pmatrix}}{\lra}
S[-1] \oplus S[-1] \overset{\begin{pmatrix} w & x
\end{pmatrix}}{\lra}  S \to S/(wS + xS) \to 0,
\]
where elements of the free modules are thought of as column vectors
and the maps are left multiplication by the indicated matrices. Then
this is an exact sequence of right $S$-modules, $W(q)S = wS + xS =
I(q)$, and this sequence is a free resolution of $P(q) = S/I(q)$. In
particular, $wS \cap xS = wyS = xzS$.
\end{lemma}
\begin{proof}
This is a special case of \cite[Proposition 6.7(ii)]{ATV2}.
\end{proof}

Using the lemmas above, we now prove an important combinatorial
formula.   In fact, it includes as a special case the calculation of
the Hilbert function of the ring $R(D)$ from the introduction, in
case $D = p$ is a single point.   As usual, we make the convention
on binomial coefficients that $\D \binom{n}{k} = 0$ if $n < k$.
\begin{proposition}
\label{ps-product-lem} \label{ps-product-prop} Let $W = W(p)
\subseteq S_1$ be any point space. Then for all $m \geq 0$ and $n
\geq 0$ we have
\[
  \dim_k S_m (WS_2)^n =
\dim S_{m+3n} -\binom{n+1}{2}.
\]
\end{proposition}
\begin{proof}
Let $V(m,n) = S_m (WS_2)^n$, and put $h(m,n) = \dim_k V(m,n)$ and $
j(m, n) = \dim_k S_{m+3n} -\binom{n+1}{2}$.  Our goal is to prove
that $h(m,n) = j(m,n)$ for all $m, n \geq 0$.  Note that since $
\dim_k S_{m+3n} = \binom{m+3n+2}{2}$ is known, it is a triviality to
check that $j$ satisfies the following recurrence relations:
\begin{gather}
\label{jeq1}
j(m,n) = j(m,n-1)  +  3m + 8n, \\
\label{jeq2} j(m,n) = 2 j(m+2, n-1) - j(m+4, n-2),
\end{gather}
for all $m \geq 0, n \geq 2$.

Now we check the formula for small $n$.  First, $h(m,0) = j(m,0)$ is
obvious for any $m \geq 0$.  Next, suppose that $n = 1$; then
$V(m,1) = S_m W(p) S_2 = W(\sigma^{-m}(p)) S_{m+2}$ by
Lemma~\ref{ps-basic-lem}(1).  Since $S/W(\sigma^{-m}(p)) S$ is a
right point module for $S$ by Lemma~\ref{pt-syz-lem}, we must have
$h(m,1) = \dim_k W(\sigma^{-m}(p)) S_{m+2} = \dim_k S_{m + 3} - 1 =
j(m,1)$, as required.

Since $S/WS$ is a point module and $g$ annihilates all point
modules, we must have $g \in WS_2$. Thus for $n \geq 1$, the kernel
of the quotient map $V(m,n) \to \overline{V(m,n)}$ contains
$gV(m,n-1)$. In this case we have $\dim_k \overline{V(m,n)} = 3m +
8n$, using Proposition~\ref{inv-section-prop} and the Riemann-Roch
theorem. Thus we also get
\[
h(m,n) \geq  h(m,n-1) + 3m + 8n,
\]
for any $n \geq 1$.  Using \eqref{jeq1} and induction on $n$, it
follows that $h(m,n) \geq j(m,n)$ for all $m,n \geq 0$.

We finish the proof by another induction on $n$.  Let $n \geq 2$ and
suppose that $h(m,k) = j(m,k)$ for all $m \geq 0$ and $k < n$.  We
have
\[
V(m,n) = S_m (W(p)S_2)^n = W(\sigma^{-m}(p)) S_{m+2} (WS_2)^{n-1} =
w U + x U,
\]
where $W(\sigma^{-m}(p)) = kw + kx$ and $U =  V(m+2, n-1)$. By
Lemma~\ref{pt-syz-lem}, we have $wy + xz = 0$, where $ky + kz =
W(\sigma^{-m-2}(p))$, and in fact $wS \cap xS = wyS = xzS$.  Then
\[
w U \cap x U = \{wys | s \in S, ys \in U, zs \in U \} = \{ wys | s
\in S, W(\sigma^{-m-2}(p))s \subseteq U \} \supseteq wy V(m+4, n-2),
\]
where the last inclusion holds since by Lemma~\ref{ps-basic-lem}(1),
\[
W(\sigma^{-m-2}(p)) S_{m+4} (WS_2)^{n-2} = S_{m+2} W S_2
(WS_2)^{n-2} = V(m+2, n-1) = U.
\]
Thus we have
\begin{gather*}
h(m,n) = \dim_k (w  U + x U ) = 2 \dim_k  U  - \dim_k (w U \cap x U)
\leq 2 h(m+2, n-1) - h(m+4, n-2) \\
= 2 j(m+2, n-1) - j(m+4, n-2) = j(m,n),
\end{gather*}
using the induction hypothesis and \eqref{jeq2}.  Since we have
already proved that $h(m,n) \geq j(m,n)$, the result follows.
\end{proof}

The last result of this section contains some other formulas about
products of point spaces, which we give here because the proofs are
also easy consequences of Lemma~\ref{pt-syz-lem}.  However, we
remark that these formulas are used only much later, in
Sections~\ref{example-sec} and \ref{deg1-sec}, and are needed not in
the proofs of the main theorems of the paper.
\begin{lemma}
\label{ps-less-basic-lem}
\begin{enumerate}

\item $\dim_k W(p)W(q)S_1 = 8$ if $q \neq \sigma^{-2}(p)$, while $\dim_k W(p)W(\sigma^2(p))S_1 =
7$.  In particular, $g \in W(p)W(q)S_1$ if and only if $q \neq
\sigma^{-2}(p)$.

\item $g \in W(p)^3$ for any point $p$.

\item Let $0 \neq f \in S_1$.  Then $g \in S_1fS_1$.
\end{enumerate}
\end{lemma}
\begin{proof}
All three parts are proved the same way.  Given $X \subseteq S_2$,
consider $W(p)X$ for some point space $W(p) = kw + kx$.  Then $W(p)X
= wX + xX$ and so $\dim_k W(p)X = 2 \dim_k X - \dim_k wX \cap xX$.
But if $y$ and $z$ are the basis of $W(\sigma^{-2}(p))$ such that
$wy + xz = 0$, as in Lemma~\ref{pt-syz-lem}, then
\[
wX \cap xX = \{wyr | yr, zr \in X \} = wy \{r | W(\sigma^{-2}(p))r
\in X \}.
\]
Thus $\dim_k wX \cap xX = \dim_k Y$, where $Y = \{r \in S_1 |
W(\sigma^{-2}(p))r \in X \}$.  In the following, since $S$ equals
$\overline{S}$ in degree up to $2$, we identify these low degree
elements of the two rings.  We also use
Lemma~\ref{inv-section-lem}(1) many times below without further
comment, to calculate the dimension of products in $\overline{S}$.

(1)  Take $X = W(q)S_1$ above.  Then $\dim_k X = 5$.  If $q =
\sigma^{-2}(p)$, then $Y = S_1$, and so $\dim_k W(p)X = 10-3 = 7$.
If instead $q \neq \sigma^{-2}(p)$, then $Y = W(\sigma(q))$, and so
$\dim_k W(p)X = 10-2 = 8$.  In any case we have $\dim_k
\overline{W(p)}\overline{W(q)}\overline{S_1} = 7$. So $g \in
W(p)W(q)S_1$ if and only if $q \neq \sigma^{-2}(p)$.

(2) Take $X = W(p)^2$ above.  Then $\dim_k X = 4$ and $\dim_k
\overline{W(p)^3} = 6$.  For a section in $S_1 = \HB^0(E, \mc{L})$
to be in $Y$, it must vanish at both $p$ and $\sigma^{-1}(p)$, so
$\dim_k Y = 1$.  Thus $\dim_k W(p)^3 = 8 - 1 = 7$, and so $g \in
W(p)^3$.

(3). We have that $f \in \HB^0(E, \mc{L})$ vanishes along some
divisor $D$.  Letting $p$ be any point such that $D$ does not
contain $\sigma^{-2}(p)$, then we claim that we even have $g \in
W(p)fS_1$.  We let $X = fS_1$ above, so $\dim_k X = 3$. Suppose that
$0 \neq y \in Y$, so that $W(\sigma^{-2}(p))y \subseteq fS_1
\subseteq \HB^0(E, \mc{L}_2)$. Since $\sigma^{-2}(p)$ is not a point
of $D$, $y$ must vanish along all of $\sigma(D)$; but since $y \in
\HB^0(E, \mc{L})$, this implies that $\sigma(D) \sim D$, a
contradiction.  So $Y = 0$ and $\dim_k W(p)fS_1 = 6$. Since $\dim_k
\overline{W(p)fS_1} = 5$, $g \in W(p)fS_1$.
\end{proof}

\section{The rings $R(D)$ and their first properties}
\label{R-sec}

We maintain the notation from the previous section, so that $S$ is a
generic Sklyanin algebra, and $\overline{X}$ indicates the image of
a subset $X \subseteq S$ under the map $S \to S/Sg$. Now let $T =
S^{(3)} = \bigoplus_{n \geq 0} S_{3n}$ be the $3$-Veronese of $S$.
Then $g \in T_1$ and $\overline{T} = T/Tg \cong B(E, \mc{M}, \tau)$,
where $\mc{M} = \mc{L}_3 = \mc{L} \otimes \mc{L}^{\sigma} \otimes
\mc{L}^{\sigma^2}$ has degree $9$ on $E$ and $\tau = \sigma^3$.  We
recall the definition of the ring $R(D)$, which was given in the
introduction. Let $D$ be an effective divisor on $E$ with $0 \leq e
= \deg D \leq 7$. Let $\mc{N} = \mc{I}_D \otimes \mc{M}$, where
$\mc{I}_D = \mc{O}_E(-D)$.
We define $V = V(D) = \{x \in T_1 | \overline{x} \in \HB^0(E,
\mc{N}) \}$.     Note that $\overline{V} = \HB^0(E, \mc{N})$
satisfies $\dim_k \overline{V} = 9 -e$ by the Riemann-Roch theorem,
and so $\dim_k V = 10- e$. Let $R = R(D) = k \langle V \rangle$ be
the $k$-subalgebra of $T$ generated by $V$.  The notation of this
paragraph is used throughout the rest of the paper, except in
Sections~\ref{thcr-factor-sec}-\ref{divisor-sec}, where we study a
slightly more general class of rings.

The first step in our analysis of $R(D)$ is to identify the ring
$\overline{R} = R/R \cap Tg \subseteq \overline{T}$.  Since $R$ is
generated in degree $1$, $\overline{R}$ is the subalgebra of $B(E,
\mc{M}, \tau)$ generated in degree $1$ by $\HB^0(E, \mc{N})$. By
Lemma~\ref{inv-section-lem}(2), since $\deg \mc{N} \geq 2$,
$\overline{R}$ is simply the subring $B(E, \mc{N}, \tau)$ of
$\overline{T}$.  In particular, using Riemann-Roch, the Hilbert
series of $\overline{R}$ is
\begin{equation}
\label{Rbar-hs-eq} \HS_{\overline{R}}(t) = 1 + \sum_{n \geq 1}
(9-e)n t^n = \frac{t^2 + (7-e)t + 1}{(1-t)^2}.
\end{equation}
Using the help of Proposition~\ref{ps-product-lem} from the last
section, we now calculate the Hilbert series of $R$.
\begin{theorem}
\label{HS-thm}  Fix an effective divisor $D$ on $E$ with $0 \leq e =
\deg D \leq 7$, and let $R = R(D)$. Then
\begin{enumerate}
\item $\HS_R(t) = \D \frac{t^2 + (7-e)t + 1}{(1-t)^3}$.
\item $R \cap Tg = Rg$.
\end{enumerate}
\end{theorem}
\begin{proof}
In fact, the statements in parts (1) and (2) of the theorem are
equivalent.  Since $R \cap Tg \supseteq Rg$ in any case and $R$ is a
domain, we definitely have $\HS_{\overline{R}}(t) \leq
(1-t)\HS_R(t)$. Part (2) holds if and only if this is an equality of
Hilbert series, which is if and only if (1) holds, by
\eqref{Rbar-hs-eq}.  We thus prove both parts of the theorem in
tandem.

We first extend $k$ to an uncountable algebraically closed field if
necessary.  Since the Hilbert series of $R$ remains unchanged, this
loses no generality.

If $\deg D = 0$ there is nothing to prove, since (2) is a tautology
in this case. Now let $\deg D = 1$, say $D = p$.  Since $S/W(p)S$
has the Hilbert series of a point module, by Lemma~\ref{pt-syz-lem},
it is easy to see that $R_1 = W(p)S_2 \subseteq T_1$.   Applying
Proposition~\ref{ps-product-lem} with $m = 0$, we have
\[
\HS_R(t) = \frac{t^2 + 7t + 1}{(1-t)^3} - \frac{t}{(1-t)^3}  =
\frac{t^2 + 6t + 1}{(1-t)^3},
\]
and so (1) (and hence also (2)) holds in this case.

Next, suppose that the theorem holds for $R' = R(D')$, for some
effective divisor $D'$ with $0 \leq \deg D' < 7$, and let $D = D' +
p$. We now prove that the theorem holds for $R = R(D)$, assuming the
further restriction that $p$ does not lie on the $\tau$-orbit of any
point occurring in $D'$.
Let $R^\# = R' \cap R(p)$. Obviously $R^\# \supseteq R$; we will
show that in fact we have equality. First we find $\overline{R^\#}$.
We know that $\overline{R'} = B(E, \mc{I}_{D'} \otimes \mc{M},
\tau)$ and $\overline{R(p)} = B(E, \mc{I}_p \otimes \mc{M}, \tau)$.
Thus in degree $n \geq 1$, we have
\[
\overline{R(p)}_n = \HB^0(E, \mc{I}_C \otimes \mc{M}_n),
\]
where $C = \sum_{j=0}^{n-1} \tau^{-j}(p)$, and similarly,
\[
\overline{R'}_n = \HB^0(E, \mc{I}_{C'} \otimes \mc{M}_n),
\]
where $C' = \sum_{j = 0}^{n-1} \tau^{-j}(D')$. But the assumption on
$p$ implies that the divisors $C$ and $C'$  have disjoint support,
so it follows that we have
\[
\overline{R(p)}_n \cap \overline{R'}_n = \HB^0(E, \mc{I}_{C +C'}
\otimes \mc{M}_n) = \overline{R}_n,
\]
since $C + C' = \sum_{j = 0}^{n-1} \tau^{-j}(D)$.
But then we see that we have $\overline{R^\#} \subseteq
\overline{R(p)} \cap \overline{R'} = \overline{R} \subseteq
\overline{R^\#}$, and thus $\overline{R} = \overline{R^\#}$. Now
note that since $R(p) \cap Tg = R(p)g$ and $R' \cap Tg = R'g$, we
have
\[
R^\# \cap Tg = R(p) \cap R' \cap Tg \subseteq R(p)g \cap R'g = (R(p)
\cap R')g = R^\#g \subseteq R^\# \cap Tg.
\]
Thus all inclusions here are actually equalities, and $R^\# \cap Tg
= R^\#g$. Since $R \cap Tg \supseteq Rg$ and $\overline{R} =
\overline{R^\#}$, it follows that $\HS_R(t) \geq
\HS_{\overline{R}}(t)/(1-t) = \HS_{\overline{R^\#}}(t)/(1-t) =
\HS_{R^\#}(t)$; since $R \subseteq R^\#$, this forces $R = R^\#$.
Then $R \cap Tg = Rg$, so (2) holds for $R = R(D)$.

Finally, assume again that the theorem holds for some $D'$, where $0
\leq \deg D' < 7$, and consider $D = D' + p$ where $p \in E$ now
varies. We show that $R(D)$ has the correct Hilbert series for all
$p$; the theorem will then follow by induction on $\deg D$. We have
a regular map $\overline{\phi}: E \to \mb{P}(\HB^0(E, \mc{I}_{D'}
\otimes \mc{M})^*)$, namely the map to projective space determined
by a basis of sections for the sheaf $\mc{I}_{D'} \otimes \mc{M}$ on
$E$, which can be described explicitly as $p \mapsto \HB^0(E,
\mc{I}_{D' + p} \otimes \mc{M}) = \overline{V(D)}$. Thus the map
$\phi: E \to \mb{P}(V(D')^*)$ given by $p \mapsto V(D)$ is also
regular.  Fixing some $n \geq 1$, we define a function $f:
\mb{P}(V(D')^*) \to \mb{N}$ by the rule $Y \mapsto \dim_k Y^n$,
where $Y$ is a codimension-1 subspace of $V(D')$ and $Y^n \subseteq
T_n$ is the product inside $T$.  A standard argument now gives that
$f$ is lower semi-continuous; in other words, $\{ Y \in
\mb{P}(V(D')^*) | \dim_k Y^n \leq m \}$ is Zariski closed for each
$m \in \mb{N}$.
Thus $f \circ \phi: E \to \mb{N}$, where $f \circ \phi(p) = \dim_k
V(D)^n = \dim_k R(D)_n$, is also lower semi-continuous.  We showed
in the previous paragraph that $R(D)$ has the correct Hilbert series
for all points $p \in E$ outside of a union of finitely many
$\tau$-orbits; on the other hand, $E$ has uncountably many points
since $k$ is uncountable.   It thus follows, from the
lower-semicontinuity of $f \circ \phi$ for each $n$, that
$\HS_{R(D)}(t) \leq \D \frac{t^2 + (7-e)t + 1}{(1-t)^2}$ for all $p
\in E$, where $e = \deg D$.  The reverse inequality was essentially
shown in the first paragraph of the proof, so we are done.
\end{proof}

For later use, we need the additional Hilbert series calculation
given in the next lemma.  The proof uses exactly the same technique
as the proof of the previous result.
\begin{lemma}
\label{M-hs-lem} Let $D'$ be an effective divisor with $0 \leq \deg
D' \leq 6$, let $p \in E$ and let $D = D' + p$.  Let $e = \deg D =
\deg D' + 1$.  Let $R = R(D)$ and $R' = R(D')$.  Consider the right
$R$-module $M = k + R'_1 R$. Then
\[
\HS_M(t) = \HS_R(t) + \D \frac{t}{(1-t)^2}=  \frac{(8-e)t +
1}{(1-t)^3}.
\]
\end{lemma}
\begin{proof}
An easy calculation, using Lemma~\ref{inv-section-lem}(1), shows
that $M/(Tg \cap M) = \overline{M} = \overline{k + R_1' R}$ has
Hilbert series $\D \frac{(8-e)t + 1}{(1-t)^2}$. Since multiplication
by $g$ is an injective homomorphism of $M$, we have
\[ \HS_M(t) = \D \frac{\HS_{M/Mg}(t)}{(1-t)} \geq
\frac{\HS_{\overline{M}}(t)}{(1-t)} = \frac{(8-e)t + 1}{(1-t)^3}. \]

Now suppose that $p$ does not lie on the $\tau$-orbit of any point
in the support of $D'$. Then we showed in the course of the proof of
Theorem~\ref{HS-thm} that $R' \cap R(p) = R$.  Since all Hilbert
series of the terms in this equation are now known by
Theorem~\ref{HS-thm}, it easily follows that $R' + R(p) = T$.  This
certainly implies that $R' + (k + T_1R(p)) = T$.  But the Hilbert
series of $k + T_1 R(p)$ follows from
Proposition~\ref{ps-product-prop}, so this determines immediately
the Hilbert series of $N = R' \cap (k + T_1R(p))$, which a
calculation shows to be $\hs_N(t) = \D \frac{(8-e)t + 1}{(1-t)^3}$.
Since $M \subseteq N$, the estimate of the first paragraph forces $M
= N$, and thus the Hilbert series of $M$ is as stated.

Now just as in the proof of the previous theorem, we can reduce to
the case where $k$ is uncountable.  A completely analogous argument
as that of the last paragraph of the previous proof also shows that
for any $n \geq 1$, the $k$-dimension of $M_n = R'_1 R_{n -1}$ is a
lower semi-continuous function of $p \in E$.  The estimate $\HS_M(t)
\geq \D \frac{(8-e)t + 1}{(1-t)^3}$ holds for all $p$, and was shown
above to be an equality for all but countably many $p$; this forces
$\HS_M(t) = \D \frac{(8-e)t + 1}{(1-t)^3}$ for all $p$, as required.
\end{proof}

We close this section with a few more properties of $R(D)$ that
follow easily at this point.  Together with Theorem~\ref{HS-thm},
the following result finishes the proof of
Theorem~\ref{RD-props-thm}(1).
\begin{theorem}
\label{Rop-thm} Let $R = R(D)$ for some effective divisor $D$ with
$0 \leq \deg D \leq 7$.
\begin{enumerate}
\item $R$ has infinite global dimension.

\item  $Q_{\on{gr}}(R) = Q_{\on{gr}}(T)$.

\item $R^{op} \cong R(\wt{D})$ for some divisor $\wt{D}$ with $\deg \wt{D} = \deg D$.
\end{enumerate}
\end{theorem}
\begin{proof}
(1) A standard argument (for example, see \cite[Proposition
2.9]{ATV2}), using the minimal graded free resolution of $k_R =
R/R_{\geq 1}$, shows that if $R$ has finite global dimension, then
$\HS_R(t) =\D \frac{1}{p(t)}$ for some polynomial $p(t) \in
\mb{Z}[t]$. But it is easy to prove that the Hilbert series of $R$,
as calculated in Theorem~\ref{HS-thm}, cannot be written in this
form.

(2) If $D = 0$ there is nothing to prove, so assume $\deg D > 0$ and
write $D = D' + p$ for some divisor $D'$. We will show that
$Q_{\on{gr}}(R) = Q_{\on{gr}}(R')$, where $R = R(D)$ and $R' =
R(D')$.  Then the result follows by induction on $\deg D$. By
Lemma~\ref{M-hs-lem}, $\dim_k R'_1R_1 - \dim_k R_2  = 2$.  We may
write $R'_1 = R_1 + kz$ for some element $z \in R'_1$, and so $R'_1
R_1 = zR_1 + R_2$.  Since $\dim_k R_1 \geq 3$ (given that $\deg D
\leq 7$), we must have $z R_1 \cap R_2 \neq 0$.  Then $z \in
Q_{\on{gr}}(R)$, so $R' \subseteq Q_{\on{gr}}(R)$ and thus
$Q_{\on{gr}}(R) = Q_{\on{gr}}(R')$.

(3) Consider the ring $B = B(E, \mc{M}, \tau)$ and the set of all
possible subspaces of the form $\HB^0(E, \mc{I}_D \otimes \mc{M})
\subseteq B_1$ for divisors with $0 \leq \deg D \leq 7$.  We claim
that this set can also be characterized as the set of all subspaces
$W \subseteq B_1$ such that the ring $k \langle W \rangle$ satisfies
$Q_{\rm gr}(k \langle W \rangle) = Q_{\rm gr}(B)$ and $\dim_k W^n =
n(\dim_k W)$ for all $n \geq 1$. In fact, the claim follows easily
from Lemma~\ref{B-subring-lem}(3). As a consequence, it is easy to
see that any graded anti-automorphism $\phi: B \to B$ must permute
this set of subspaces.

Now, it is well known that $S^{op} \cong S$; for example, the map
induced by $x \mapsto y, y \mapsto x, z \mapsto z$ gives an
anti-automorphism $\phi: S \to S$, as one easily sees from the
presentation.  Passing to the $3$-Veroneses, this induces an
anti-automorphism $\phi: T \to T$.  Since the central element $g$ is
unique up to scalar, $\phi$ must descend to an anti-automorphism
$\psi: T/Tg \to T/Tg$, where $T/Tg \cong B = B(E, \mc{M}, \tau)$. By
the first paragraph, $\psi$ must permute the set of subspaces of the
form $\HB^0(E, \mc{I}_D \otimes \mc{M}) \subseteq B_1$.  Thus $\phi:
T \to T$ permutes the subspaces of the form $V(D) \subseteq T_1$ for
divisors $0 \leq \deg D \leq 7$. So for given $D$, $\phi$ restricts
to an anti-isomorphism from $R(D)$ to $R(\wt{D})$ for some other
divisor $\wt{D}$, which clearly must have the same degree.
\end{proof}

\section{Rings with a twisted homogenous coordinate ring factor}
\label{thcr-factor-sec}

In this section, we prove that the rings $R(D)$ have many additional
nice properties; in particular, we prove
Theorem~\ref{RD-props-thm}(2).  Since we now know that $R(D)/R(D)g
\cong B(E, \mc{N}, \tau)$, the main technique is to observe that
many good properties pass from the factor ring $B$ to the ring
$R(D)$.  This technique is well-known and holds much more generally.
Because it may be useful for future reference, and is not really any
more difficult, in this section and the next two we work not just
with the rings $R(D)$, but with the more general class of rings
satisfying the following hypothesis.
\begin{hypothesis}
\label{main-hyp} Let $R$ be a cfg $k$-algebra which is generated in
degree $1$.  Suppose that $R$ is a domain with a homogeneous central
element $g$ of degree $d > 0$ such that $R/Rg \cong B(E, \mc{N},
\tau)$, where $E$ is a nonsingular elliptic curve with fixed
basepoint $p_0$, $\tau: E \to E$ is a translation automorphism $x
\mapsto x + s$ for a point $s$ of infinite order, and $\deg \mc{N}
\geq 2$. We continue to use the notation $\overline{Y}$ for the
image of a subset $Y$ of $R$ under the map $R \to R/Rg$.
\end{hypothesis}
\noindent
Hypothesis~\ref{main-hyp} is assumed throughout the rest of this
section.  We remark that for some of the results of this and the
next few sections, it is sufficient that $g$ be a normal element
$(Rg = gR)$; in particular, this is true of
Theorems~\ref{basic-prop-thm} and \ref{max-order-thm} below.



The following definition is standard in this setting.
\begin{definition}
\label{g-tors-def}  Given any cfg algebra $A$ with a central element
$g$ and $M \in \rgr A$, we define its \emph{$g$-torsion submodule}
$\ts(M) = \{m \in M | mg^n = 0\ \text{for some}\ n > 0 \}$. $M$ is
\emph{$g$-torsion} if $\ts(M) = M$ and \emph{$g$-torsionfree} if
$\ts(M) = 0$.
\end{definition}
\noindent This torsion theory has the usual formal properties.  For
example, it is easy to verify that as long as $A$ is noetherian,
then for any $M \in \rgr A$, $\ts(M) g^n = 0$ for some $n \geq 1$,
and $M/\ts(M)$ is $g$-torsionfree.

In the next theorem, we give a number of properties of rings
satisfying Hypothesis~\ref{main-hyp} which follow quite immediately
from the properties of twisted homogeneous coordinate rings studied
earlier, together with results from the literature about when good
properties ``pass up" from a factor ring modulo a central (or
normal) element.
\begin{theorem}
\label{basic-prop-thm} Let $R$ satisfy Hypothesis~\ref{main-hyp}.
\begin{enumerate}
\item $R$ is strongly noetherian.

\item Every $M \in \rgr R$ has a Hilbert quasi-polynomial of period
$d$. 

\item $R$ is Auslander-Gorenstein and Cohen-Macaulay.

\item $R$ satisfies $\chi$ on the left and right, and $\proj R$
has cohomological dimension $2$.
\end{enumerate}
\end{theorem}
\begin{proof}
(1) Since $B(E, \mc{N}, \tau)$ is strongly noetherian by
Lemma~\ref{B-prop}(1), $R$ is strongly noetherian by
\cite[Proposition 4.9(1)]{ASZ}.

(2) For any module $M \in \rgr R$, consider its $g$-torsion
submodule $\ts(M)$.   Since $\ts(M)g^n = 0$ for some $n$, $\ts(M)$
has a filtration with finitely many factor modules
$\ts(M)g^i/\ts(M)g^{i+1}$, each of which is a graded $B$-module;
thus $\ts(M)$ has a Hilbert polynomial, by Lemma~\ref{B-prop}(3).
Since $M' = M/\ts(M)$ is $g$-torsionfree, $\D H_{M'}(t) =
\frac{H_{(M'/M'g)}(t)}{(1-t^d)}$. Here, $M'/M'g$ again is a
$B$-module and so has a Hilbert polynomial.  The form of the Hilbert
series of $M'$ then implies that $M'$ has a Hilbert quasi-polynomial
of period $d$.  Thus $M$ does also.

(3)  The ring $B(E, \mc{N}, \tau)$ is Auslander-Gorenstein of
dimension $2$ and Cohen-Macualay, by Lemma~\ref{B-prop}(4).  Then
$R$ is Auslander-Gorenstein of dimension $3$ and Cohen-Macaulay, by
\cite[Theorem 5.10]{Lev}.

(4) We know that $B$ satisfies $\chi$ and $\cd(\rproj B) = 1$, by
Lemma~\ref{B-prop}(2), so $R$ satisfies $\chi$ by \cite[Theorem
8.8]{AZ1}, and $\cd(\rproj R) \leq 2$ by \cite[Theorem 8.8]{AZ1}. We
need to show equality in the last statement.  The higher cohomology
in $\rproj R$ can be calculated as a limit of Ext groups:
\[
\bigoplus_{m \in \mb{Z}} H^i(\pi(M)[m]) = \lim_{n \to \infty}
\ext^{i+1}_R(R/R_{\geq n}, M),
\]
for any $i \geq 1$ and $M \in \rgr R$ \cite[Proposition 7.2]{AZ1}.
Here, the direct limit is induced by the obvious factor maps
$R/R_{\geq (n+1)} \to R/R_{\geq n}$.  Now since $R$ is
Cohen-Macaulay by part (3), for any $N$ with $\GK N = 0$ we have
$\ext^i_R(N,R) = 0$  if $i < 3$, but $\ext^3_R(N,R) \neq 0$. It
follows from the long exact sequence in Ext that the maps in the
direct limit $\lim_{n \to \infty} \ext^3(R/R_{\geq n}, R)$ are
injective, and thus $\bigoplus_{m \in \mb{Z}} H^2(\pi(R)[m]) \neq
0$.  So $\cd(\rproj R) = 2$ exactly.
\end{proof}
Next, we study homogeneous ideals of rings $R$ satisfying
Hypothesis~\ref{main-hyp}.  Here are a few basic properties of such
ideals.
\begin{lemma} Let $R$ satisfy Hypothesis~\ref{main-hyp}.
\label{special-ideal-lem}
\begin{enumerate}
\item If $J \subseteq R$ is a nonzero homogeneous ideal, then $J = Ig^n$ for
some $n \geq 0$ and some ideal $I$ with $\GK R/I \leq 1$.
\item If $I$ is a homogeneous ideal of $R$ with $\GK R/I \leq 1$, then
there is a unique homogeneous ideal $I' \supseteq I$ such that
$\dim_k I'/I < \infty$ and $R/I'$ is $g$-torsionfree.
\item If $M \in \rgr R$ is $g$-torsionfree with $\GK M \leq 1$, then $I =
\rann_R(M)$ is an ideal such that  $\GK R/I = 1$ and $R/I$ is
$g$-torsionfree.
\end{enumerate}
\end{lemma}
\begin{proof}
(1)   Since $J \neq 0$, there is a maximal $n \geq 0$ such that $J
\subseteq Rg^n$. Then $J = Ig^n$ for some homogeneous ideal $I$.
Since $I \nsubseteq Rg$ by choice of $n$, $\overline{I} \neq 0$ in
$R/Rg = B(E, \mc{N}, \tau)$, which is a projectively simple ring by
Lemma~\ref{B-subring-lem}(1).  Thus $\wt{R} = R/I$ satisfies $\dim_k
\wt{R}/\wt{R}g < \infty$, from which it follows that $\GK \wt{R}
\leq 1$.

(2)  Let $I$ be such an ideal. Since $\GK R/Rg = 2$, $I \nsubseteq
Rg$ and so the same argument as in part (1) shows that $\dim_k R/(I
+ Rg) < \infty$. But then since $(I + Rg^i)/(I + Rg^{i+1})$ is a
finitely generated $R/(I + Rg)$-module for each $i$, $\dim_k R/(I +
Rg^n) < \infty$ for each $n$.  Write $I'/I = \ts((R/I)_R)$; clearly
$I'$ is a homogeneous ideal.  Then $I'/I$ is annihilated by some
$g^n$, and thus is a finitely generated $R/(I + Rg^n)$-module. So
$\dim_k I'/I < \infty$, as claimed.

(3)
Since $M$ has a Hilbert quasi-polynomial of period $d$, by
Theorem~\ref{basic-prop-thm}(2), $\GK M \leq 1$ implies that $M$ has
eventually periodic Hilbert function of period $d$.  In particular,
since $M$ is $g$-torsionfree, $Mg$ is equal in large degree to $M$,
so there is $n_1$ such that  $(Mg)_{\geq n_1} = M_{\geq n_1}$. Now
since $g$ is central, we conclude that $M_n = (M_{n-d}) g$ and
$M_{n-d}$ have the same right annihilator in $R$, for all $n \geq
n_1$. Thus $I = \rann_R M = \rann_R (\bigoplus_{n = n_0}^{n_1-1}
M_n)$, where $n_0$ is an integer such that $M_n = 0$ for $n < n_0$.
But since $\bigoplus_{n = n_0}^{n_1-1} M_n$ is a finite-dimensional
subspace of the module $M$ of GK-dimension $\leq 1$, this forces
$\GK R/I \leq 1$ as needed.  It is easy to check that $R/(\rann_R
M)$ is automatically $g$-torsionfree, because $M$ is.
\end{proof}

We see from part (1) of the preceding result that we will understand
all of the ideals of $R$ if we understand the ideals $I$ such that
$\GK R/I = 1$.  It is convenient to make the following definitions.
\begin{definition}  Let $A$ be a cfg $k$-algebra.  If $I$ is a homogeneous ideal of $A$
such that $\GK A/I  \leq 1$, then we call $I$ \emph{special} and we
call $A/I$ a \emph{special factor ring}.  A special ideal $I$ of $A$
is said to be \emph{minimal in tails} if $\pi(I)$ is (uniquely)
minimal among the set of objects $\{ \pi(J) \in \rqgr A | J\
\text{is a special ideal of}\ A \}$.  In other words, this means
that every other special ideal $J$ of $A$ satisfies $J \supseteq
I_{\geq n}$ for some $n \geq 0$.  If $A$ has a special ideal which
is minimal in tails, then (speaking loosely) we say that \emph{$A$
has a minimal special ideal.}
\end{definition}
\noindent The reason for the terminology ``special" is that the
rings $S$ and $T$ have no special factor rings except those of
finite $k$-dimension, as we recall in Lemma~\ref{T-ideal-lem} below.
The rings $R(D)$ sometimes do have GK-1 factor rings (see
Example~\ref{main-ex}), but they are rather rare, exactly in the
sense that $R(D)$ always has a minimal special ideal. Proving this
is the goal of much of the next several sections.

\begin{remark}
\label{special-rem} Suppose that $R$ satisfies
Hypothesis~\ref{main-hyp}, and that $R$ has a minimal special ideal.
Using the results of Lemma~\ref{special-ideal-lem} we can say more
about this situation. Since clearly any two ideals which are both
minimal in tails agree in large degree, using
Lemma~\ref{special-ideal-lem}(2) there is a \emph{unique} ideal $I$
which is minimal in tails with the additional property that $R/I$ is
$g$-torsionfree.  In other words, $R$ has a unique largest
homogeneous $g$-torsionfree special factor ring $R/I$.  Also, by
Lemma~\ref{special-ideal-lem}(3), it is then easy to see that this
same ideal $I$ is also the unique largest ideal which annihilates
all GK-1 $g$-torsionfree right $R$-modules.  Because of
Theorem~\ref{Rop-thm}(3), $I$ must also be the unique largest ideal
annihilating all GK-1 $g$-torsionfree left $R$-modules.
\end{remark}


We now show that rings satisfying the main hypothesis of this
section are maximal orders, as a quick consequence of the
Cohen-Macaulay property.     That the maximal order property is
closely connected to such homological properties more generally is
well-known; for example, see \cite{St}.  Together with
Theorem~\ref{basic-prop-thm}, this completes the proof of
Theorem~\ref{RD-props-thm}(2).
\begin{theorem}
\label{max-order-thm}
\item Let $R$ satisfy Hypothesis~\ref{main-hyp}.  Then
$R$ is a maximal order in its Goldie quotient ring.
\end{theorem}
\begin{proof}
For a homogeneous ideal $J$ of $R$, we define $\mc{O}_r(J) =
\Hom_R({} _R J, _R J) =  \{ x \in Q_{\on{gr}}(R) | Jx \subseteq J
\}$ and $\mc{O}_{\ell}(J) = \Hom_R( J_R, J_R) = \{ x \in
Q_{\on{gr}}(R) | xJ \subseteq J \}$.   To show that $R$ is a maximal
order, it suffices to show that for all homogeneous ideals $J$ of
$R$, we have $\mc{O}_r(J) = \mc{O}_{\ell}(J) = R$ \cite[Lemma
9.1]{Ro}.  We prove that $\mc{O}_{\ell}(J) = R$; the proof that
$\mc{O}_{r}(J) = R$ is symmetric, and is left to the reader.

Any homogeneous ideal $J$ of $R$ has the form $J = J'g^n$ where $J'$
is a special ideal, by Lemma~\ref{special-ideal-lem}(1). Then
\[
\mc{O}_{\ell}(J) = \{x \in Q_{\rm gr}(R)| xJ'g^n \subseteq J'g^n \}
= \{x \in Q_{\rm gr}(R) | xJ' \subseteq J' \} = \mc{O}_{\ell}(J').
\]
So it will be enough to consider special homogeneous ideals $J$, and
show that $\Hom_R(J_R, J_R) = R$.  Since $\Hom_R(J_R, J_R) \subseteq
\Hom_R(J_R, R_R)$, it is enough to show that $\Hom_R(J_R, R_R) = R$.
But we have the following portion of the long exact sequence in Ext:
\[
\dots \lra R = \Hom_R(R_R, R_R) \lra \Hom_R(J_R, R_R) \lra
\ext^1_R(R/J,R) \lra \dots
\]
Since $\GK(R/J) \leq 1$, we have $\ext^1_R(R/J,R) = 0$ by the
Cohen-Macaulay property of $R$, which holds by
Theorem~\ref{basic-prop-thm}(3). Thus $\Hom_R(J_R, R_R) = R$ and the
result follows.
\end{proof}

\section{Subrings}
\label{subrings-sec} In order to prove our main classification
result, Theorem~\ref{deg3-class-thm}, we need to get some
understanding of when a cfg subring $A$ of a ring $R(D)$ is an
equivalent order to $R(D)$.  Thus, in this section we study when
this happens for subrings $A$ which are close to $R(D)$, in the
sense that $\overline{A} \subseteq \overline{R(D)}$ is a finite ring
extension.  This is a natural assumption, because it is what we will
actually know when we apply these results later in the proof of
Theorem~\ref{deg3-class-thm}.

As in the previous section, we actually work with a slightly more
 general setup.  For the rest of this section, we fix the
 following notation.
\begin{hypothesis}
\label{subs-hyp} Let $R$ be a ring satisfying
Hypothesis~\ref{main-hyp}, and let $A \subseteq R$ be a cfg
subalgebra such that $Q_{\gra}(A) = Q_{\gra}(R)$ and $\overline{A}
\subseteq \overline{R}$ is a finite ring extension. Let $C = A
\langle g \rangle$ be the subring generated by $A$ and the central
element $g \in R_d$, so that $A \subseteq C \subseteq R$.
\end{hypothesis}

We first make a convenient definition concerning ring extensions.
Let $U \subseteq V$ be an inclusion of cfg $k$-algebras which are
domains with $Q_{\gra}(U) = Q_{\gra}(V)$.  We say that $U$ and $V$
\emph{have a common ideal} if there is a homogeneous ideal $0 \neq
I$ of $V$ such $I \subseteq U$.  Clearly if $U$ and $V$ have a
common ideal, then they are equivalent orders.  Here are some other
simple facts about this concept.
\begin{lemma}
\label{ideal-close-lem} Let $U \subseteq V \subseteq W$ be cfg
$k$-algebras which are domains with $Q_{\gra}(U) = Q_{\gra}(V) =
Q_{\gra}(W)$.
\begin{enumerate}

\item If $U \subseteq V$ is a finite ring extension, then $U$ and $V$ have a common ideal.

\item If $U$ and $V$ have a common ideal and $V$ and $W$ have a common
ideal, then $U$ and $W$ also have a common ideal.
\end{enumerate}
\end{lemma}
\begin{proof}

(1) This argument may be found in the proof of \cite[Theorem
4.1]{AS}, but we repeat it here for convenience.  Since $U$ and $V$
have the same graded quotient ring, every element of $(V/U)_U$ has
nonzero annihilator in $U$. Moreover, $_U (V/U)_U$ is a bimodule
which is finitely generated on both sides. It follows that $I =
\rann_U (V/U)$ is a nonzero homogeneous right ideal of $U$ which
also satisfies $VI \subseteq I$.  A symmetric argument gives that $J
= \lann_U(V/U)$ is a nonzero homogeneous left ideal of $U$ such that
$JV \subseteq J$. Thus $0 \neq IJ$ is a homogeneous ideal of $V$
which is contained in $U$.

(2)  Let $0 \neq I$ be a common ideal for $U$ and $V$ and $0 \neq J$
a common ideal for $V$ and $W$.  Then $0 \neq JIJ$ is a common ideal
for $U$ and $W$.
\end{proof}



As mentioned earlier, one of our main concerns is to understand when
the ring $A$ in Hypothesis~\ref{subs-hyp} is an equivalent order to
$R$.  We will prove this below only under a further assumption that
$R$ has a minimal special ideal.  The case where $g \in A$, in other
words the case of the ring $C$, is much more straightforward, and so
we study it separately first.
\begin{lemma}
\label{C-noeth-lem}
Assume Hypothesis~\ref{subs-hyp}.
\begin{enumerate}
\item $C \subseteq R$ is a finite extension and $C$ is a noetherian ring.

\item Every special factor ring of $C$ has eventually periodic
Hilbert function of period $d$.  Also, if $R$ has a minimal special
ideal, then so does $C$.

\end{enumerate}
\end{lemma}
\begin{proof}
(1) The graded Nakayama lemma shows that $R_C$ is a finitely
generated module if and only if $\dim_k R/R C_{\geq 1} < \infty$.
Since $g \in C_{\geq 1}$, this is also equivalent, by Nakayama
again, to $\overline{R} = R/Rg$ being a finitely generated right
$\overline{C}$-module, which is true by hypothesis. A similar
argument holds on the left, so $C \subseteq R$ is a finite
extension.

It remains to prove that $C$ is noetherian.  We claim that in fact
the following fact holds: given a finitely generated $\mb{Z}$-graded
right $C$-module $M$, if $(M/Mg)_C$ is noetherian then $M_C$ is
noetherian.
First note that since $M$ is finitely generated, $M$ is right
bounded ($M_n = 0$ for $n \ll 0$).  Using this, it is easy to prove
that if $M$ has a non-finitely generated submodule, it has a
homogeneous non-finitely generated submodule.  Since $g$ is
obviously still a central (and hence normal) element of $C$, the
rest of the proof of the claim follows by exactly the same argument
as \cite[Lemma 8.2]{ATV1}, which proves the case where $M = C$.

Apply the claim to $M = R$; this shows that $R_C$ is noetherian as
long as $(R/Rg)_C$ is noetherian.  In fact, $\overline{R} = R/Rg$ is
a $\overline{C} = C/(C \cap Rg)$ module, and so it is enough to
prove that $\overline{R}_{\overline{C}}$ is noetherian.  This
follows since $\overline{C} \subseteq \overline{R}$ is a finite ring
extension, and $\overline{C}$ is noetherian by
Lemma~\ref{B-subring-lem}(1).  Thus $R_C$ is noetherian, and a
symmetric argument shows that $_C R$ is noetherian.  Of course,
then, $C$ is a noetherian ring.

(2) 
Suppose that $I$ is an arbitrary special ideal of $C$.   Since $C
\subseteq R$ is a finite extension by part (1),
Lemma~\ref{ideal-close-lem}(1) implies that $C$ and $R$ have a
common nonzero homogeneous ideal, say $0 \neq H$. We can assume that
$H \subseteq I$, by replacing $H$ by $HIH$ if necessary. By
Lemma~\ref{special-ideal-lem}(1), $H = Kg^n$ for some special ideal
$K$ of $R$ and $n \geq 0$. By Lemma~\ref{special-ideal-lem}(2), $K$
has a finite dimensional extension $L$ such that $R/L$ is
$g$-torsionfree; by increasing $n$ if necessary, we can take $H =
Lg^n$ as our common ideal.

Now given any $i \geq 1$, notice that the right $C$-module $M = ( (C
\cap Lg^{i-1}) + I)/((C \cap Lg^i) + I)$ is killed by $(C \cap Rg) +
I$, so $M$ is a module over $C/((C \cap Rg )+ I) =
\overline{C}/\overline{I}$.  The ring $\overline{C}$ is projectively
simple, by Lemma~\ref{B-subring-lem}(1).  Also, $\overline{I} \neq
0$ since $\GK C/I = 1$, so necessarily $\dim_k
\overline{C}/\overline{I} < \infty$. We know that $C$ is noetherian
by part (1), so $M \in \rgr C$ and thus $\dim_k M < \infty$.   Since
$(C \cap Lg^n) + I = I$, we conclude by induction on $i$ that
$\dim_k ((C \cap L) + I)/I < \infty$, and so $I \supseteq (C \cap
L)_{\geq m}$ for some $m$.

To prove that $C/I$ has eventually $d$-periodic Hilbert function, we
might as well now replace $I$ by $I + (C \cap L)$, so we can assume
that $C/I$ is a factor of $C/(C \cap L)$.  Now $N = C/(C \cap L)$ is
$g$-torsionfree and has bounded Hilbert function, since $R/L$ has
these properties. An easy argument shows then that $\dim_k N/Ng <
\infty$.  Then $N' = C/I$ is a surjective image of $N$ and so also
$\dim_k N'/N'g < \infty$.  Then another easy argument shows that
this forces $N'$ to have eventually $d$-periodic Hilbert function,
as needed.

Finally, if $R$ has a minimal special ideal, then as in
Remark~\ref{special-rem} there is a unique ideal $L'$ of $R$ which
is minimal in tails and such that $R/L'$ is $g$-torsionfree.  Then
necessarily $L' \subseteq L$, and so we can replace $L$ in the above
argument by $L'$. Then any special ideal $I$ of $C$ has $I \supseteq
(L' \cap C)_{\geq m}$ for some $m$, and this implies that $L' \cap
C$ is an ideal of $C$ which is  minimal in tails.
\end{proof}

Next, we handle the more complicated case of subrings of $R$ that
need not contain $g$.  We now need to impose the existence of a
minimal special ideal to gain more traction.
\begin{theorem}
\label{subring-thm} Assume Hypothesis~\ref{subs-hyp}, and assume in
addition that $R$ has a minimal special ideal.
\begin{enumerate}
\item $A$ and $R$ have a common ideal; in particular, $R$ and $A$
are equivalent orders.
\item If $A$ is noetherian, then $A \subseteq R$ is a finite
extension.
\end{enumerate}
\end{theorem}
\begin{proof}
(1)  We have $C = A \langle g \rangle = A + Ag + Ag^2 + \dots$. Note
for each $n \geq 0$ that $CA_{\geq n} = A_{\geq n} + A_{\geq n}g +
A_{\geq n}g^2 + \dots = A_{\geq n}C$. In particular, $CA_{\geq n}$
is an ideal of $C$. Moreover, the Hilbert series of $C/CA_{\geq n}$
is at most as large as $\sum_{i \geq 0} p(t)t^{id}$, where $p(t) =
\hs_{A/A_{\geq n}}(t)$ is a polynomial since $\dim_k A/A_{\geq n} <
\infty$.  So $\GK C/CA_{\geq n} \leq 1$, and $CA_{\geq n}$ is a
special ideal of $C$. Now $C$ has a minimal special ideal, and all
special factor rings of $C$ have eventually $d$-periodic Hilbert
function, by Lemma~\ref{C-noeth-lem}(2).  It is easy then to see
that any descending chain of special ideals of $C$ must eventually
stabilize in tails.  Thus $\dim_k CA_{\geq n_0}/CA_{\geq n} <
\infty$ for all $n \geq n_0$, some $n_0$.

We claim that $CA_{\geq n_0}$ is a finitely generated right
$A$-module. By the graded Nakayama Lemma, to show this it suffices
to show that $\dim_k CA_{\geq n_0}/CA_{\geq n_0}A_{\geq 1} <
\infty$. Though we don't know that $A$ is noetherian, we do have
that $A$ is a finitely generated $k$-algebra by hypothesis.  If $e$
is the maximum degree of its generators, then clearly $A_{\geq (e +
n_0)} \subseteq A_{\geq n_0}A_{\geq 1}$.  So it is enough to show
that $\dim_k CA_{\geq n_0}/CA_{\geq (e + n_0)} < \infty$, which we
already showed above; the claim is proved.  A symmetric proof shows
that $CA_{\geq n_0}$ is finitely generated on the left over $A$. In
conclusion, setting $\wt{A} = A + CA_{\geq n_0}$, then $A \subseteq
\wt{A}$ is a finite ring extension.

We have constructed a series of ring extensions as follows, where
all rings have the same graded quotient ring:
\begin{equation}
\label{ring-ext-eq} A \subseteq \wt{A}= A + CA_{\geq n_0} \subseteq
C \subseteq R.
\end{equation}
Here, both of the outer extensions $A \subseteq \wt{A}$ and $C
\subseteq R$ are finite ring extensions (using
Lemma~\ref{C-noeth-lem}(1)), and thus have a common ideal, by
Lemma~\ref{ideal-close-lem}(1). Also, the extension $A + CA_{\geq
n_0} \subseteq C$ certainly has a common ideal, namely $CA_{\geq
n_0}$. By Lemma~\ref{ideal-close-lem}(2), the extension $A \subseteq
R$ has a common ideal, and $A$ and $R$ are equivalent orders.

(2) Suppose that $A$ is noetherian, and consider the argument in
part (1) above and the chain of ring extensions \eqref{ring-ext-eq}.
Since $\wt{A} = A+ CA_{\geq n_0}$ is finite over $A$ on both sides,
it is also a noetherian ring in this case.  Choosing any nonzero
homogeneous element $x \in A_{\geq n_0}$, we see that $xC \subseteq
A$ and so $C_{A}$ embeds in $A$; thus $C_{A}$ is Noetherian.
Similarly, $_{A} C$ is noetherian.  Thus $A \subseteq C$ and $C
\subseteq R$ are finite ring extensions, and so $A \subseteq R$ is
also a finite extension.  In fact, it is easy to see that $\dim_k
C/\wt{A} < \infty$ in this case.
\end{proof}

If in Hypothesis~\ref{subs-hyp} we have that $\overline{A} \subseteq
\overline{R}$ is not just a finite extension but an equality in
large degrees, then one can expect to say more about how close $A$
is to $R$.  The following result studies this. There is a strong
result for the ring $C$, and a rather weaker one for a general $A$.
We note that this result is not needed in the proof of
Theorem~\ref{deg3-class-thm}, though it is used in
Section~\ref{deg1-sec} below in the classification of subrings of
$S$ generated in degree $1$. We also do not know if the extra
hypotheses in part (1) of the following proposition are necessary.
\begin{proposition}
\label{close-prop}
Assume the setup in Hypothesis~\ref{subs-hyp},
and suppose that $\dim_k \overline{R}/\overline{A} < \infty$.
\begin{enumerate}
\item If in addition $A \subseteq R$ is a finite extension and
$(A \cap Rg)/(A \cap Rg^2) \neq 0$, then $A$ contains a special
ideal of $R$.
\item $C$ contains a special ideal of $R$.
\end{enumerate}
\end{proposition}
\begin{proof}
(1)  Since $A \subseteq R$ is a finite extension with $Q_{\gra}(A) =
Q_{\gra}(R)$, $A$ and $R$ have a common ideal by
Lemma~\ref{ideal-close-lem}(1).  This ideal must have the form
$Jg^n$ for some special ideal $J$ of $R$, by
Lemma~\ref{special-ideal-lem}(1). $J$ has a finite-dimensional
extension $J'$ such that $R/J'$ is $g$-torsionfree
(Lemma~\ref{special-ideal-lem}(2)), so possibly after increasing $n$
we can replace $J$ by $J'$ and assume that $R/J$ is $g$-torsionfree.

 We have the following exact sequence:
\[
0 \to (J \cap Rg)/Jg \to J/Jg \to R/Rg \to R/(J + Rg) \to 0.
\]
Here, $(J \cap Rg)/Jg \cong (M/J)[-d]$, where $M = \{x \in R | xg
\in J \}$.  Then $M/J$ is contained in the $g$-torsion submodule
$\ts(R/J)$ of $R/J$.  By choice of $J$,  $M = J$ and so $(Rg \cap
J)/Jg = 0$.  Thus $J/Jg$ injects into $R/Rg$, and this is even an
inclusion of $\overline{R}-\overline{R}$-bimodules.  For the
purposes of this proof, following the definition for rings, we call
a bimodule $M$ \emph{projectively simple} if every sub-bimodule $0
\neq N$ of $M$ satisfies $\dim_k M/N < \infty$. Since $R/Rg$ is a
projectively simple $\overline{R}-\overline{R}$-bimodule, by
Lemma~\ref{B-subring-lem}(1), clearly $J/Jg$ is also a projectively
simple $\overline{R}-\overline{R}$-bimodule.  Since multiplication
by $g$ induces a bimodule isomorphism $Jg^{i-1}/Jg^{i} \to
Jg^{i}/Jg^{i+1}$ for each $i$,  $Jg^i/Jg^{i+1}$ is also a
projectively simple $\overline{R}-\overline{R}$-bimodule. Since
$\dim_k \overline{R}/\overline{A} < \infty$, it is then easy to see
that $Jg^i/Jg^{i+1}$ is also a projectively simple
$\overline{A}-\overline{A}$-bimodule for each $i$.

By hypothesis, we can choose $x \in A \cap Rg \setminus Rg^2$.
Clearly we can also choose $y \in A \cap J \setminus Rg$, since we
have $\GK A/(A \cap J) \leq 1$, while $\GK A/A \cap Rg = \GK
\overline{A} = 2$. Then $x^iy \in A \cap Jg^i \setminus Rg^{i+1}$,
and thus $(A \cap Jg^i)/(A \cap Jg^{i+1}) \neq 0$.  Since $(A \cap
Jg^i)/(A \cap Jg^{i+1})$ is a $\overline{A}-\overline{A}$
sub-bimodule of $Jg^i/Jg^{i+1}$, we conclude that $(A \cap Jg^i)/(A
\cap Jg^{i+1})$ equals $Jg^i/Jg^{i+1}$ in large degree.

Now recall that $Jg^n \subseteq A$.  So there exists an $i \geq 0$
minimal with the property that $(Jg^i)_{\geq m} \subseteq A$ for
some $m \geq 0$.  Suppose $i > 0$.  Since $(A \cap Jg^{i-1})/(A \cap
Jg^i)$ is equal in large degree to $Jg^{i-1}/Jg^i$, and $A \cap Jg^i
\supseteq (Jg^i)_{\geq m}$ for some $m$, we conclude that $A \cap
Jg^{i-1}$ is equal in large degree to $Jg^{i-1}$. Then
$(Jg^{i-1})_{\geq m'} \subseteq A$ for some $m'$, contradicting the
choice of $i$. So $i = 0$, and $A$ contains a special ideal of $R$
of the form $J_{\geq m}$.

(2)  We have that $C \subseteq R$ is a finite extension by
Lemma~\ref{C-noeth-lem}(1), and clearly $(C \cap Rg)/(C \cap Rg^2)
\neq 0$ since $g \in C$. Taking $A = C$ in part (1), the result is
immediate.
\end{proof}

\section{Divisors of modules}
\label{divisor-sec} The concept of a divisor associated to a GK-2
module is a basic tool in the study of the module category of the
Sklyanin algebra $S$, and AS-regular algebras more generally.  It
was first studied by Ajitabh in \cite{Aj1}, and then also by Van den
Bergh, De Naeghel, and others, for example see \cite{Aj2},
\cite{AjV}, \cite{DeN}.  We will see that this concept is also very
important for the study of the rings $R(D)$. In this section, we
recall some of the basic definitions and properties of divisors of
modules.  At the end of the section, we apply the theory we have
developed to line modules and their submodules; as it turns out,
line modules will also be very important to the proof of
Theorem~\ref{RD-props-thm}(3).
Throughout this section, let $R$ be a ring satisfying
Hypothesis~\ref{main-hyp}, so that $g \in R_d$ is a central element
(many of the results work with minor adjustments if $g$ is normal),
and $B = R/Rg = B(E, \mc{N}, \tau)$.

We call $M \in \rgr R$ \emph{admissible} if $\GK M/Mg \leq 1$. Since
$B$ is a GK-2 domain and $M/Mg \in \rgr B$, this is
equivalent to the requirement that $M/Mg$ does not have a subfactor
isomorphic to (a shift of) $B$.  We will associate divisors only to
admissible modules. For admissible $M$, tensoring the exact sequence
$0 \to Rg \to R \to R/Rg \to 0$ with $M$ and considering the long
exact sequence in $\tor$, we easily see that
\begin{equation}
\label{tor1-eq} \tor_1^R(M, R/Rg) \cong \{ m \in M | mg = 0 \}[-d],
\end{equation}
and that $\tor_i^R(M, R/Rg) = 0$ for all $i \geq 2$. We now define
the divisor associated to an admissible module; this is a variation
on the original definition of Ajitabh, following Van den Bergh
\cite[p. 66]{VdB}, which has the advantage of exactness as in
Lemma~\ref{div-basic-lem}(4) below.
\begin{definition}
\label{div-def} Given $M \in \rgr R$ which is admissible, let $N =
M/Mg$ and $N' =\tor_1^R(M, R/Rg)$.  We define $\dv M = C(N) - C(N')
\in \Div E$, in the notation of Definition~\ref{C-def}.
\end{definition}
For the preceding definition to make sense, we need $\GK \tor_1^R(M,
R/Rg) \leq 1$ when $M$ is admissible.  This follows from
\eqref{tor1-eq} and part (1) of the next lemma, which also collects
several other easy properties of the divisor.
\begin{lemma}
\label{div-basic-lem} Let $M \in \rgr R$.
\begin{enumerate}
\item $M$ is admissible if and only if $\GK M \leq 2$ and
$\GK \ts(M) \leq 1$.
\item If $M$ is admissible and $g$-torsionfree, then $\dv M =
C(M/Mg)$; in particular, $\dv M$ is effective.
\item For admissible $M$ and $i \in \mb{Z}$, $\dv M[i] = \tau^{i}(\dv
M)$.
\item The map $M \mapsto \dv M$ is additive on short exact sequences of
admissible modules in $\rgr R$.
\end{enumerate}
\end{lemma}
\begin{proof}
(1) Applying $- \otimes R/Rg$ to the exact sequence $0 \to \ts(M)
\to M \to M' = M/\ts(M) \to 0$, one obtains an exact sequence
\[
\dots \to \tor^1(M', R/Rg) \to \ts(M)/\ts(M)g \to M/Mg \to M'/M'g
\to 0.
\]
Since $M'$ is $g$-torsionfree, in fact $\tor^1(M', R/Rg) = 0$.  Now
if $\GK M \leq 2$ and $\GK \ts(M) \leq 1$, then $\GK \ts(M)/\ts(M)g
\leq 1$.   Then since $\GK M' \leq 2$ and $M'$ is $g$-torsionfree,
$\GK M'/M'g \leq 1$ by a simple Hilbert series argument.  So $\GK
M/Mg \leq 1$ from the exact sequence.  Conversely, if $\GK M/Mg \leq
1$, then $\GK \ts(M)/\ts(M)g \leq 1$ from the exact sequence. Since
$\ts(M)$ has a finite filtration $\ts(M) \supseteq \ts(M)g \supseteq
\dots \supseteq \ts(M) g^n = 0$ for some $n$, and
$\ts(M)g^i/\ts(M)g^{i+1}$ is a (shifted) surjective image of
$\ts(M)g^{i-1}/\ts(M)g^i$ for each $i$, we see that $\GK \ts(M) \leq
1$.  Also, $\GK M/Mg \leq 1$ implies that $\GK M \leq 2$, by Hilbert
series again.

(2) This is immediate from the definition of $\dv M$ and
\eqref{tor1-eq}.

(3) It follows from a definition chase that for $N \in \rgr B$ with
$\GK N \leq 1$, we have $C(N[i]) = \tau^{i}(C(N))$, which implies
the result.

(4) Apply $- \otimes R/Rg$ to a given exact sequence $0 \to M \to N
\to P \to 0$ of admissible modules.  Then the desired result follows
immediately from the corresponding long exact sequence in $\tor$
(recall that $\tor_2^R(P, R/Rg) = 0$), together with the fact that
$C(-)$ is additive on short exact sequences of GK-1 modules in $\rgr
B$.
\end{proof}

The divisor $\dv M$ is an important tool in understanding
$g$-torsionfree modules $M \in \rgr R$ with $\GK M = 2$.  Namely, to
get a picture of such modules $M$, one first tries to understand
which linear equivalence classes of divisors occur as $\dv M$. Such
linear equivalence classes have been completely characterized when
$R = S$ is the generic Sklyanin algebra, by work of De Naeghel
\cite{DeN}. The complete answer is quite intricate, but in the
sequel we only need the weak consequence in the following lemma,
which already follows from Ajitabh's first paper \cite{Aj1}.
\begin{definition} Let $R$ satisfy Hypothesis~\ref{main-hyp}.  Then we define
\[
\dv_2(R) = \{ \dv M | M \in \rgr R\ \text{is}\ g\text{-torsionfree}\
\text{and}\ GK(M) = 2 \} \subseteq \Div E.
\]
\end{definition}

\begin{lemma}
\label{Skl-count-lem}
  Let $S$ be a generic Sklyanin algebra.  Then the image of $\dv_2(S)$
in $\Pic E$ is countable.
\end{lemma}
\begin{proof}
Every $g$-torsionfree $M \in \rgr S$ with $\GK M = 2$ has an
extension $M' \supseteq M$ with $\dim_k M'/M < \infty$, where $M'$
has projective dimension $1$ (see the discussion following
\cite[Lemma 2.5]{Aj1}).  Then by \cite[Lemma 2.5]{Aj1}, the linear
equivalence class of $\dv M = \dv M'$ is uniquely determined by the
sequence of graded Betti numbers appearing in the minimal projective
resolution of $M'$, and there are clearly countably many such
sequences of Betti numbers. (Technically, Ajitabh uses a different
definition of the divisor $\dv M$ in this paper, which it is
well-known is equivalent to ours for the modules $M$ in question
\cite[p. 1636]{AjV}.)
\end{proof}

While it would be interesting to study the divisors of modules over
$R(D)$ more carefully, our goal in this paper is more modest: we
just want to prove the analog of Lemma~\ref{Skl-count-lem} for the
rings $R(D)$, which we will do in the next section.  It simplifies
arguments considerably to work modulo the divisors of $\GK$-1
modules.  This is accomplished by the next definition, following an
idea of Van den Bergh \cite[p. 111]{VdB}.
\begin{definition}
Let $H_{\tau}$ be the subgroup of $\Div E$ generated by $\{ p -
\tau^i(p) | p \in E, i \in \mb{Z} \}$.  Let $H_{\ell}$ be the
subgroup of $\Div E$ of divisors linearly equivalent to $0$, so that
$\Pic E = \Div E/H_{\ell}$ as usual.  We define $\overline{\Div} E =
\Div E/H_{\tau}$, and $\overline{\Pic} E = \Div E/(H_{\ell} +
H_{\tau})$.  For $M \in \rgr R$, let $\overline{\dv} M$ be the image
in $\overline{\Div} E$ of $\dv M$. Let $\overline{\dv}(R) = \{
\overline{\dv} M | M \in \rgr R\ \text{is admissible} \} \subseteq
\overline{\Div} E$.
\end{definition}
\begin{lemma}
\label{over-dv-lem}
\begin{enumerate}
\item If $M \in \rgr R$ with $\GK M \leq 1$, then $\overline{\dv} M =
0$. In particular, if $M \in \rgr R$ is admissible then
$\overline{\dv}(M) = \overline{\dv}(M/\ts(M))$.

\item The image of $\dv_2(R)$ in $\Pic E$ is countable if and only
if the image of $\overline{\dv}(R)$ in $\overline{\Pic E}$ is
countable.

\end{enumerate}

\end{lemma}
\begin{proof}
(1) Note that $\overline{\dv}(-)$ will still be additive on short
exact sequences, using Lemma~\ref{div-basic-lem}(4).  So to prove
that if $M \in \rgr R$ with $\GK M \leq 1$, then $\overline{\dv} M =
0$, it suffices to prove this when $M$ is either $g$-torsion or
$g$-torsionfree.  If $M$ is $g$-torsionfree, then $\GK M/Mg = 0$,
and so $\dv M = 0$.  If $M$ is $g$-torsion, then it has a filtration
by finitely many $B$-modules of GK-dimension at most $1$. If $N \in
\rgr B$, then $N' = \tor_1^R(N, R/Rg) \cong N[-d]$, and so $\dv N =
C - \tau^{-d}(C)$ for some effective divisor $C$; clearly then
$\overline{\dv}(N) = 0$ by definition.  The second statement
follows, using Lemma~\ref{div-basic-lem}(1).

(2) Let $\theta: \Pic E \to \overline{\Pic E}$ be the quotient map.
For any point $p$ and $i \in \mb{Z}$, we have $p - \tau^i(p) \sim
-is$, where $s$ is the point such that $\tau$ is the translation
$\tau(x) = x + s$ in the group structure.   Thus $H_{\ell} +
H_{\tau} = H_{\ell} + \mb{Z}s$.  It follows that $\theta$ is a
countable-to-1 map.  Now let $P$ be the image of $\dv_2(R)$ in $\Pic
E$.  Since by part (1) only $g$-torsionfree modules of GK-2 matter
in determining which divisors occur in $\overline{\dv}(R)$, it is
easy to see that $\theta(P) \cup \{0\}$ is equal to the image of
$\overline{\dv}(R)$ in $\overline{\Pic E}$.  The result follows,
since $P$ is countable if and only if $\theta(P)$ is countable.
\end{proof}

Next, we study how the divisors of modules over $R$ are related to
divisors of modules over a Veronese ring. Suppose that $m \geq 1$
divides the degree $d$ of the central element $g$ of $R$. Then the
$m$th Veronese ring $R' = R^{(m)} = \bigoplus_{n \geq 0} R_{nm}$
still satisfies the basic setup of this section: setting $d' = d/m$,
we have $g \in R'_{d'}$ is central in $R'$ and $R'/R'g \cong B^{(m)}
\cong B(E, \mc{N}_m, \tau^m)$.  Given a module $M \in \rgr R$, we
can define a module $M^{(m)} = \bigoplus_{n = 0}^{\infty} M_{nm} \in
\rgr R'$; the rule $M \mapsto M^{(m)}$ defines a functor $F: \rgr R
\to \rgr R'$.
\begin{lemma}
\label{div-ver-lem} Keep the notation of the preceding paragraph.
Given any admissible $M \in \rgr R$, then $\dv_R M = \dv_{R'}
M^{(m)}$.  Moreover, $\overline{\dv}(R) = \overline{\dv}(R')$.
\end{lemma}
\begin{proof}
First, suppose that $N \in \rgr B$ is isomorphic in large degree to
$\bigoplus_{n = 0}^{\infty} \HB^0(E, \mc{F} \otimes \mc{N}_n)$ for
some torsion sheaf $\mc{F}$. Then $N^{(m)}$ is isomorphic as a $B' =
R'/R'g = B(E, \mc{N}_m, \tau^m)$-module to $\bigoplus_{n
=0}^{\infty} \HB^0(E, \mc{F} \otimes \mc{N}_{mn})$ in large degree.
Thus in the notation of Definition~\ref{C-def}, $C_B(N) =
C_{B'}(N^{(m)})$.

Given $M \in \rgr R$, we have $M^{(m)}/M^{(m)}g = (M/Mg)^{(m)}$.
Since all $B$-modules and $B'$-modules have a Hilbert polynomial of
the form $f(n) = an + b$ for $n \gg 0$, by Lemma~\ref{B-prop}(3), it
is easy to check then that $\GK_R M/Mg = \GK_{R'} M^{(m)}/M^{(m)}g$;
thus $M$ is admissible if and only if $F(M) = M^{(m)}$ is.  Now
given an admissible $M \in \rgr R$, setting $P = \{m \in M | mg =
0\}$, we also have
\[
\tor_1^{R'}(M^{(m)}, R'/R'g)  \cong P^{(m)}[-d/m] \cong P[-d]^{(m)}
\cong \tor_1^R(M, R/Rg)^{(m)}.
\]
Thus $\dv_R M = \dv_{R'} M^{(m)}$ follows from the previous
paragraph.

Finally, because $R$ is generated in degree $1$,  the functor $F$
descends to an equivalence of categories $\overline{F}: \rqgr R \to
\rqgr R'$ \cite[Proposition 5.10(3)]{AZ1}.  In particular, any
admissible $M' \in \rgr R'$ is isomorphic in large degree to $F(M)$
for some $M \in \rgr R$, which must also then be admissible, as we
have seen. The conclusion  $\overline{\dv}(R) = \overline{\dv}(R')$
follows.
\end{proof}

Recall that a module $M \in \rgr R$ is called a \emph{line module}
if $M$ is cyclic, generated in degree $0$, and $\hs_M(t) =
1/(1-t)^2$. In the last result of this section, we completely
describe the divisors of such a line module and its submodules,
under the further condition that the central element $g$ is in
degree $1$. In this case, by Theorem~\ref{basic-prop-thm}(2), every
module $N \in \rgr R$ has a Hilbert polynomial, and so a
well-defined multiplicity.
\begin{lemma}
\label{nearline-prop-lem} \label{line-prop-lem} Let $R$ satisfy
Hypothesis~\ref{main-hyp}, and in addition assume that $g \in R_1$
$(d = 1)$.  Let $M \in \rgr R$ be a line module for $R$.
\begin{enumerate}
\item $M$ is $g$-torsionfree, admissible, and
$\dv M = p$ is a single point.

\item If $0 \neq N \subseteq M$ is a graded submodule, then there is
a graded submodule $\wt{N}$ with $N \subseteq \wt{N} \subseteq M$
such that $\dim_k \wt{N}/N < \infty$ and $\wt{N} \cong L[-m]$ for
some line module $L$.  In particular, $M$ is GK-2-critical.
Moreover, $m$ is the multiplicity of $M/N$ and $\dv L =
\tau^{m-n}(p)$, where $n$ is the multiplicity of $\ts(M/N)$.
\end{enumerate}
\end{lemma}
\begin{proof}

(1) Suppose first that $\GK \ts(M) = 2$.  Then $M$ must contain a
subfactor isomorphic to a shift of $B$.  However, since $\dim_k B_n
= (\deg \mc{N}) n$ for $n \gg 0$ by Riemann-Roch, any shift of $B_B$
is a module with multiplicity $\deg \mc{N} \geq 2$, whereas $M$ has
multiplicity $1$ by the definition of a line module.  This
contradiction shows that $\GK \ts(M) \leq 1$ and thus $M$ is
admissible by Lemma~\ref{div-basic-lem}(1).

Now $M' = M/\ts(M)$ is a $g$-torsionfree module with $\GK M' = 2$.
Note that $\GK M'/M'g = 1$.  Since $R$ is generated in degree $1$,
the smallest possible Hilbert series of a GK-1 module which is
cyclic and generated in degree $0$ is the Hilbert series of a point
module, $1/(1-t)$.  Thus $H_{M'/M'g}(t) \geq 1/(1-t)$. Since $M'$ is
$g$-torsionfree, $H_{M'}(t) \geq 1/(1-t)^2$, forcing $M' = M$ and
$\ts(M) = 0$.  This also shows that $M/Mg \in \rgr B$ actually is a
point module, and $\dv M = p$ is a single point as claimed.

(2) Suppose that $N$ is a nonzero homogeneous submodule of $M$. Let
$n \geq 0$ be maximal such that $N \subseteq Mg^n$.  Then $Ng^{-n}
\subseteq M$.  Let $Ng^{-n} \subseteq K \subseteq M$ be chosen so
that $K/Ng^{-n}$ is the largest finite-dimensional submodule of
$M/Ng^{-n}$.

By choice of $n$, $K \nsubseteq Mg$.  Since $M$ is $g$-torsionfree,
$\tor_1^R(M, R/Rg) = 0$ and so we have an exact sequence
\[
0 \to \tor_1^R(M/K, R/Rg) \to K/Kg \overset{\theta}{\to} M/Mg \to
M/(K + Mg) \to 0.
\]
By assumption, $\theta \neq 0$.   As we saw in the proof of part
(1), $M/Mg$ is a point module.  Since $R$ is generated in degree
$1$, necessarily $\im \theta = (M/Mg)_{\geq m'}$ for some $m' \geq
0$.  In particular, this forces $\GK K/Kg = 1$, and since $K$ is
$g$-torsionfree, $\GK K = 2$.  Since $M$ has Hilbert function
$\dim_k M_n = n+1$, there is no choice but for $K$ to have a Hilbert
polynomial of the form $\dim_k K_n = n + b$ for $n \gg 0$, for some
$b \leq 1$.  Then $\dim_k (K/Kg)_n = 1$ for $n \gg 0$. Since $\im
\theta$ is also equal to a point module in large degree, this forces
$\ker \theta = \tor_1^R(M/K, R/Rg)$ to be finite-dimensional over
$k$. On the other hand, $\tor_1^R(M/K, R/Rg)$ is (a shift of) the
largest submodule of $M/K$ killed by $g$; by choice of $K$, $\ker
\theta = \tor_1^R(M/K, R/Rg) = 0$.

Thus $K/Kg \cong (M/Mg)_{\geq m'}$ is a tail of a point module,
generated in degree $m'$.  Since $K$ is $g$-torsionfree, $K$ then
has the Hilbert function of a line module shifted by $-m$; since $K$
is also generated in degree $m'$ by the graded Nakayama lemma, we
must have $K \cong L[-m']$ for some line module $L$.  Setting
$\wt{N} = Kg^n$, we have $\dim_k \wt{N}/N < \infty$ and  $\wt{N}
\cong L[-m' - n]$. In particular, it certainly follows from Hilbert
series that $\GK M/N \leq 1$, and so $M$ is $\GK$-2-critical.

Since we proved that $\tor_1^R(M/K, R/Rg) = 0$, we know that
$M/K$ is $g$-torsionfree with $GK M/K \leq 1$, whereas clearly $K/N$
is $g$-torsion by construction.  It follows that $\dv M = \dv K$,
and so $\dv K = p$ and thus $\dv L = \tau^{m'}(p)$, by
Lemma~\ref{div-basic-lem}(3). The number $n$ is clearly the
multiplicity of $Ng^{-n}/N$, so also the multiplicity of $K/N =
\ts(M/N)$, and $m'$ is the multiplicity of $M/K$.  Thus $m = m' + n$
is the multiplicity of $M/N$.
\end{proof}

\section{The exceptional line module for $R(D)$}
\label{except-sec}

Beginning in this section, we refocus attention on the specific
rings $R(D)$ of interest.  For the rest of the paper, we use the
notation introduced in Sections~\ref{sklyanin-sec} and \ref{R-sec}.
Fix an effective divisor $D'$ on $E$ with $0 \leq \deg D' \leq 6$,
and let $D = D' + p \in \Div E$ for some point $p \in E$. Throughout
this section, let $R = R(D)$ and $R' = R(D')$. By
Theorem~\ref{HS-thm}, $\overline{R} = R/Rg = B(E, \mc{N}, \tau)$ and
$\overline{R}' = R'/R'g = B(E, \mc{N}', \tau)$, where $\mc{N} =
\mc{I}_{D} \otimes \mc{M}$ and $\mc{N}' = \mc{I}_{D'} \otimes
\mc{M}$.  As usual, we write $\overline{Y}$ for the image of any
subset $Y \subseteq S$ under the homomorphism $S \to S/Sg$.

Our goal is to study the ring extension $R  \subseteq R'$, and show
that intuitively the quasi-scheme $\rQgr R$ behaves as a
noncommutative blowup of $\rQgr R'$ at the point $p$.
In particular, we identify a line module in $\rgr R$ which behaves
as an exceptional object for this blowup.
We also study the relationship between divisors of modules over the
two rings $R$ and $R'$.  This will also allow us to show, by
induction on $\deg D$, that $\dv_2(R)$ is contained in countably
many linear equivalence classes of divisors.

We now construct the exceptional line module for the extension $R
\subseteq R'$. It is closely related to the structure of $(R'/R)_R$.
\begin{lemma}
\label{exc-line-lem} Fix $R \subseteq R'$ as above.
\begin{enumerate}
\item $J = \{y \in R | (R'_1)y \in R \}$ is a right ideal of $R$ such that $L_R = R/J$ is a line
module. Moreover,  $\dv L = \tau(p)$.

\item $(R'/R)_R \cong \bigoplus_{i = 0}^{\infty} L[-i-1]$.
\end{enumerate}
\end{lemma}
\begin{proof}
(1)  Fix throughout the proof some $z \in R'_1$ such $R'_1 = R_1 +
kz$.  Clearly $J = \{x \in R | zx \in R \}$. Let $M = k + R'_1 R = R
+ R'_1 R$. Then we have
\[
M/R = zR +R/R \cong zR/(zR \cap R) \cong R/J[-1].
\]
Also, $\HS_{M/R}(t)= \D \frac{t}{(1-t)^2}$ by Lemma~\ref{M-hs-lem}.
Thus $R/J$ has the Hilbert series of a line module, and is obviously
generated in degree $0$, so it is a line module.

To compute the divisor of $L$, recall that $L/Lg = R/(Rg + J) =
\overline{R}/\overline{J}$ must be a point module.  The equation
$\overline{R'_1} \overline{J} \subseteq \overline{R}$ forces
$\overline{J} \subseteq \bigoplus_{n \geq 0} \HB^0(E,
\mc{I}_{\tau(p)} \otimes \mc{N}_n)$, by considering vanishing
conditions.  This forces the point module $P(\tau(p)) = \bigoplus_{n
\geq 0} \HB^0(E, k(\tau(p)) \otimes \mc{N}_n)$ to be a factor of
$\overline{R}/\overline{J}$, so we must have
$\overline{R}/\overline{J} \cong P(\tau(p))$ and $\dv L = \tau(p)$.

(2) Let $z$ be as in part (1).  For each $n \geq 0$, we claim that
we can choose an element $w_n \in R_n$ so that $w_n z \not \in R'_n
R_1$.  Clearly we can take $w_0 = 1$.  If we cannot choose such a
$w_n$ for some $n \geq 1$, then $R_n z \subseteq R'_n R_1$, and this
forces $R_n R'_1 \subseteq R'_n R_1$, and thus $\overline{R_n}
\overline{R'_1} \subseteq \overline{R'_n} \overline{R_1}$.  But
thinking of these as subspaces of $\overline{R'_{n+1}} = \HB^0(E,
\mc{N}'_{n+1})$ determined by certain vanishing conditions (using
Lemma~\ref{inv-section-lem}), clearly this is impossible:
$\overline{R_n} \overline{R'_1}$ is the subspace of sections
vanishing along the divisor $p + \tau^{-1}(p) + \dots +
\tau^{-n+1}(p)$, while $\overline{R'_n} \overline{R_1}$ is the
subspace of sections vanishing along $\tau^{-n}(p)$.  This proves
the claim, so we fix such elements $w_n$.

We claim that the $R$-submodules $M^{(n)} = (w_n z R + R)/R$ of
$R'/R$ for $n \geq 0$ are independent and span $R'/R$. First we show
spanning.  We claim that for each $n \geq 1$, we have $N^{(n)} =
M^{(0)} + M^{(1)} + \dots + M^{(n-1)} = (R + R'_{\leq n}R)/R$, from
which spanning will immediately follow.  The base case $n = 1$ is by
definition.  Now assume the equation holds for some $n$, and
consider $N^{(n+1)} = (R + R'_{\leq n}R + w_n zR)/R$.  We have
$\dim_k R'_{n+1} - \dim_k R'_n R_1 = 1$; to see this, note that
$\dim_k \overline{R}'_{n+1} - \dim_k \overline{R}'_n \overline{R}_1
= 1$ using Lemma~\ref{inv-section-lem}, and $R'_{n+1} \cap Tg =
R'_ng \subseteq R'_n R_1$.  Moreover, $w_n z \in R'_{n+1} \setminus
R'_n R_1$ by choice of $w_n$; thus $R'_{n+1} = R'_n R_1 + k w_n z$,
and it follows that $N^{(n+1)} =  (R + R'_{\leq (n+1)}R)/R$.

Next, note that $M^{(n)}$ is cyclic and generated in degree $n + 1$.
Then $M^{(n)} \cong (R/K^{(n)})[-n-1]$, where $K^{(n)} = \{y \in R |
w_n z y \in R \}$.  Clearly we have $J \subseteq K^{(n)}$, so
$M^{(n)}$ is isomorphic to a factor module of $L[-n-1]$. However,
Theorem~\ref{HS-thm} implies that $\HS_{R'/R}(t) = \D
\frac{t}{(1-t)^3} = \sum_{n \geq 0} \frac{t^{n+1}}{(1-t)^2}$, which
is clearly the same as the Hilbert series of $\bigoplus_{n =
0}^{\infty} L[-n-1]$. Since $\sum_{n \geq 0} M^{(n)} = R'/R$, as we
showed in the previous paragraph, the only possible conclusion is
that $K^{(n)} = J$ for all $n$, and that the submodules $M^{(n)}$
are also independent.  Thus $R'/R \cong \bigoplus_{n = 0}^{\infty}
M^{(n)} \cong \bigoplus_{n = 0}^{\infty} L[-n-1]$.
\end{proof}

\begin{definition}
\label{exc-def} We call the line module $L$ appearing in the
preceding theorem the \emph{exceptional line module} (for the
extension $R = R(D) \subseteq R(D')= R'$).  A full subcategory of an
abelian category is called a \emph{Serre} subcategory if its set of
objects is closed under subobjects, factor objects, and extensions.
Let $\wt{\mc{C}}$ be the smallest Serre subcategory of $\rGr R$
containing $L$ and closed under shifts and direct limits, and let
$\mc{C} = \pi(\wt{\mc{C}}) \subseteq \rQgr R$ be the corresponding
full subcategory of $\rQgr R$. We call $\mc{C} \subseteq \rQgr R$
the \emph{exceptional category} (again, for a fixed choice of
extension $R \subseteq R'$.)
\end{definition}

We do not spend much time studying the geometric properties of the
quasi-schemes $\rQgr R(D)$ in this paper; a lot of information about
these quasi-schemes will flow from \cite{VdB}, once the connections
with that paper are fully established.  Because it is easy, however,
we do offer one theorem which further justifies the idea that $\rQgr
R$ should be thought of as a blowup of $\rQgr R'$.  The following
result shows that the difference between $\rQgr R$ and $\rQgr R'$ is
determined entirely by the exceptional category.
\begin{theorem}
\label{qgr-blowup-thm} Let $\wt{\mc{C}}$ and $\mc{C}$ be as in
Definition~\ref{exc-def}. Let $\wt{\mc{D}}$ be the subcategory of
$\rGr R'$ consisting of all objects $N_{R'}$ such that $N_R \in
\wt{\mc{C}}$, and let $\mc{D} = \pi(\wt{\mc{D}}) \subseteq \rQgr
R'$.  Then there is an equivalence of quotient categories $\rQgr
R/\mc{C} \simeq \rQgr R'/\mc{D}$.
\end{theorem}
\begin{proof}
Define functors $F: \rGr R \to \rGr R'$ and $G: \rGr R' \to \rGr R$
by $F(M) = M \otimes_R R'$ and $G(N_{R'}) = N_R$.  We claim that $F$
descends to a functor $\overline{F}: \rQgr R/\mc{C} \to \rQgr
R'/\mc{D}$. For this, we first note that if $M \in \rGr R$, then we
have an exact sequence in $\rGr R$ as follows:
\begin{equation}
\label{tensor-seq-eq} 0 \to \tor_1^R(M, R') \to \tor_1^R(M, R'/R)
\to M \to M \otimes_R R' \to M \otimes_R (R'/R) \to 0.
\end{equation}
By Lemma~\ref{exc-line-lem}(2), clearly $(R'/R) \in \wt{\mc{C}}$. It
easily follows that $\tor^R_1(M, R'/R)$ and $M \otimes_R (R'/R)$ are
also in $\wt{\mc{C}}$.  Thus if $M \in \wt{\mc{C}}$, then $(M
\otimes_R R')_R$ is in $\wt{\mc{C}}$ also, so $F(M) \in
\wt{\mc{D}}$.  We also see from the exact sequence that $\tor_1^R(M,
R') \in \wt{\mc{D}}$ for any $M \in \rGr R$. Thus the induced
functor $\widehat{F}: \rGr R \to \rGr R'/\wt{\mc{D}}$ is an
\emph{exact} functor such that $\widehat{F}(\wt{\mc{C}}) = 0$. By
the universal property of the quotient category, $\widehat{F}$
induces a functor $\overline{F}: \rGr R/\wt{\mc{C}} \to \rGr
R'/\wt{\mc{D}}$; but since $\rTors R \subseteq \wt{\mc{C}}$ and
$\rTors R' \subseteq \wt{\mc{D}}$, this is the same thing as a
functor $\overline{F}: \rQgr R/\mc{C} \to \rQgr R'/\mc{D}$, as
claimed.   Since $G$ is already itself an exact functor, a similar
but easier argument shows that $G$ descends to a functor
$\overline{G}: \rQgr R'/\mc{D} \to \rQgr R/\mc{C}$. In addition, the
same exact sequence \eqref{tensor-seq-eq} above easily yields that
$\overline{G} \, \overline{F}$ is naturally isomorphic to the
identity functor.

Now for $N \in \rGr R'$, there is a natural map $\phi_N: FG(N) \to
N$ given by the multiplication map $\phi: N \otimes_R R' \to N$.  We
claim that this descends to give a natural isomorphism between
$\overline{F} \, \overline{G}$ and the identity functor.  Since
$\phi$ is surjective, it suffices to show that $\ker \phi \in
\wt{\mc{D}}$, or equivalently that $(\ker \phi)_R \in \wt{\mc{C}}$.
Suppose that $0 \neq s = \sum (n_i \otimes r_i) \in \ker \phi$ is a
homogenous element, where $n_i \in N$ and $r_i \in R'$ are
homogeneous.  Define for each $n \geq 1$ the right $R$-ideal
$J^{(n)} = \{y \in R | (R'_{\leq n})y \subseteq R \}$.
We can choose some $n \geq 1$ for which $r_i J^{(n)} \subseteq R$
for all $i$, and then clearly $s J^{(n)} = 0$. Thus $(\ker \phi)_R$
is a direct limit of factors of the modules $R/J^{(n)}$; it suffices
to show that $R/J^{(n)} \in \wt{\mc{C}}$ for all $n \geq 1$.
However, letting $\{v_1, \dots, v_m\}$ be a $k$-basis of $R'_{\leq
n}$, we have that $J^{(n)} = \bigcap_{i=1}^m \rann_R (v_i + R)$,
where $v_i + R \in R'/R$ for each $i$.  Since $R'/R \in
\wt{\mc{C}}$, it easily follows that each $R/\rann_R (v_i + R)$, and
hence also $R/J^{(n)}$, is in $\wt{\mc{C}}$ as needed.
\end{proof}
Intuitively, modding out the exceptional category $\mc{C}$ from
$\rQgr R$ corresponds to removing the exceptional line.  In accord
with the intuition that $\rQgr R$ should be thought of as a blowup
of $\rQgr R'$ at the point $p$, one might hope that it would be
enough to mod out from $\rQgr R'$ a subcategory generated by the
simple object corresponding to the point $p$.  However, in general
it seems that some $R'$-modules of $\GK$-2 are entangled and must
get included in $\wt{\mc{D}}$.  See
Proposition~\ref{main-ex-prop}(4) below for an example.

In the next result, we compare the divisors of admissible modules
over the two rings $R$ and $R'$.
Not surprisingly, the difference between $\overline{\dv}(R)$ and
$\overline{\dv}(R')$ is generated entirely by the divisor of the
exceptional line module.
\begin{proposition}
\label{dv2-prop} Let $R = R(D) \subseteq R' = R(D')$ as above.
\begin{enumerate}
\item Suppose that $M \in \rgr R'$ is admissible.
Then $M_R \in \rgr R$, $M_R$ is admissible, and $\dv M_R  = \dv
M_{R'}$.

\item If $M \in \rgr R$ is in the exceptional category $\wt{\mc{C}}$ of Definition~\ref{exc-def},
then $\overline{\dv} M  = np \in \overline{\Div} E$ for some $n \geq
0$.

\item $\overline{\dv}(R) \subseteq \{ G + np \, | \, G \in \overline{\dv}(R'), n \in
\mb{Z} \} \subseteq \overline{\Div} E$.
\end{enumerate}
\end{proposition}
\begin{proof}
(1) Recall that $R'/R'g = B' = B(E, \mc{N}', \tau)$ and $R/Rg = B =
B(E, \mc{N}, \tau)$.  For any $N \in \rgr B'$ with $\GK N \leq 1$,
we have $N_B \in \rgr B$, and $C(N_{B'})=C(N_B)$, by
Lemma~\ref{restrict-lem}.

Now for $M \in \rgr R'$ admissible, let $N = M/Mg \in \rgr B'$,
where $\GK N \leq 1$.  Then $N_B \in \rgr B$, and it follows from
the graded Nakayama lemma that $M_R \in \rgr R$.  Clearly  $M_R$ is
admissible.  Moreover,
\[
\tor_1^{R'}(M, R'/R'g)_R \cong \{m \in M | mg = 0 \}[-1] \cong
\tor_1^{R}(M_R, R/Rg).
\]
Thus $\dv M_R = \dv M_{R'}$ follows from the definition of the
divisor.

(2) We have that $M$ is filtered by subfactors of shifts of the
exceptional line module $L$, so since $\overline{\dv}$ is additive
on exact sequences, it suffices to prove the result when $M$ is a
subfactor of a shift of $L$.  Since $L$ is GK-2-critical by
Lemma~\ref{line-prop-lem}(2), using Lemma~\ref{over-dv-lem}(1) it
suffices to assume $M$ is a nonzero submodule of a shift of $L$. But
then $\dv M = \tau^i(p)$ for some $i \in \mb{Z}$ by
Lemma~\ref{line-prop-lem} and Lemma~\ref{exc-line-lem}(1), so
$\overline{\dv} M = p$.

(3)  Let $M \in \rgr R$ be admissible.  Consider again the exact
sequence of right $R$-modules
\[
\dots \to \tor^1_R(M, R'/R) \to M \overset{\theta}{\to} M' \to M
\otimes_R (R'/R) \to 0,
\]
where $M' = M \otimes_R R'$.  As we already observed in the proof of
Theorem~\ref{qgr-blowup-thm}, the terms $\tor^1_R(M, R'/R)$ and $M
\otimes_R (R'/R)$ must be in the exceptional category $\wt{\mc{C}}$,
and thus $\ker \theta \in \wt{\mc{C}}$ and $\im \theta \in
\wt{\mc{C}}$. Note that $M' \in \rgr R'$, and that $M'/M'g = (M/Mg)
\otimes_R R' = (M/Mg) \otimes_B B'$.  We have $\GK_B M/Mg \leq 1$,
so that $M/Mg$ is a Goldie-torsion module over the domain $B$.  It
follows that $(M/Mg) \otimes_B B'$ is Goldie-torsion over $B'$,
which forces $\GK_{B'} \big((M/Mg) \otimes_B B'\big) \leq 1$ also.
Thus $M' \in \rgr R'$ is admissible. Then by part (1), $M'_R$ is
admissible, $\overline{\dv}(M'_{R'}) = \overline{\dv} M'_R$, and
$\overline{\dv}(M_R) \subseteq \{ G + np \, | \, G \in
\overline{\dv}(R'), n \in \mb{Z} \}$ follows immediately from the
exact sequence and part (2).
\end{proof}

\begin{corollary}
\label{dv2-countable-cor}
 Let $R = R(D)$ for any effective $D$ with $0 \leq \deg D \leq
7$.  Then the image of $\dv_2(R)$ in $\Pic E$ is countable.
\end{corollary}
\begin{proof}
Let $S$ be the generic Sklyanin algebra.  By
Lemma~\ref{Skl-count-lem}, we know that the image of $\dv_2(S)$ in
$\Pic E$ is countable.  Thus the image of $\overline{\dv}(S)$ in
$\overline{\Pic} E$ is countable, by Lemma~\ref{over-dv-lem}.
Letting $T = S^{(3)}$ as usual, then since $\overline{\dv}(S) =
\overline{\dv}(T)$ by Lemma~\ref{div-ver-lem}, the image of
$\overline{\dv}(T)$ in $\overline{\Pic} E$ is countable.  It now
follows, by Proposition~\ref{dv2-prop}(3) and induction on $\deg D$,
that the image of $\overline{\dv}(R(D))$ in $\overline{\Pic} E$ is
countable. Thus the image of $\dv_2(R)$ in $\Pic E$ is countable, by
Lemma~\ref{over-dv-lem} again.
\end{proof}

As an application of the previous results, we can now prove the
following interesting result about the divisors of the line modules
of the ring $R(D)$.
\begin{theorem}
\label{line-finite-thm} Let $R = R(D)$ where $0 \leq \deg D \leq 7$.
Then the set $\{ \dv L | L\ \text{is a line module for}\ R \}$ is
finite.
\end{theorem}
\begin{proof}
First, we reduce to the case of an uncountable base field. Let $k
\subseteq \ell$ be a field extension where $\ell$ is uncountable and
algebraically closed.  If $M \in \rgr R$ is a line module, then $M
\otimes_k \ell$ is a line module over $R_{\ell} = R \otimes_k \ell$,
which is just $R(D)$ constructed over the bigger field $\ell$. Also,
identifying the points of $E$ with a subset of $E \times_{\spec k}
\spec \ell$, $\dv M_R = \dv (M \otimes_k \ell)_{R_{\ell}}$. Thus we
see that it is enough to prove the result over the base field
$\ell$.  Using this, we assume for the rest of the proof that $k$ is
uncountable.

Now since $R$ is strongly noetherian, by
Theorem~\ref{basic-prop-thm}(1), the work of Artin and Zhang shows
that there is a projective scheme $X$, the \emph{line scheme}, which
parameterizes the line modules for $R$. We recall its construction:
for any fixed degree $m$, one considers the set of all graded
subspaces $J_0 \oplus J_1 \oplus \dots \oplus J_m \subseteq R_0
\oplus R_1 \oplus \dots \oplus R_m$ such that $\dim_k (R/J)_n = n+1$
for all $0 \leq n \leq m$, and $J_i R_j \subseteq J_{i+j}$ for all
$i, j$ with $i + j \leq m$.  This is a closed subscheme of a product
of projective Grassmannians, and so is a projective scheme $X_m$.
Now \cite[Corollary E4.5]{AZ2} shows that there is some finite $m$
such that any such $\bigoplus_{n = 0}^m J_n$ generates a right ideal
$J$ of $R$ such that $R/J$ is a line module; then for this $m$, $X =
X_m$ is the line scheme.

Recall that the divisor of a line module over $R$ is a single point,
by Lemma~\ref{line-prop-lem}.  We claim that given any point $q \in
E$, the set of line modules $M$ such that $\dv M = q$ forms a closed
subset of $X$, which we write as $X_q$.  Given a line module $M =
R/J$, the condition $\dv M = q$ is equivalent to $R/J + Rg \cong
P(q)$, where $P(q) \in \rgr B$ is the point module associated to the
point $q$. But there is a unique right ideal $I \supseteq Rg$ of $R$
such that $R/I \cong P(q)$, so it is also equivalent to demand that
$\bigoplus_{n = 0}^m J_n \subseteq \bigoplus_{n = 0}^m I_n$. This is
clearly a further closed condition in the construction above,
proving the claim.

Now Corollary~\ref{dv2-countable-cor} shows that the image of
$\dv_2(R)$ in $\Pic E$ is countable.  Since every line module $M$ is
GK-2 and $g$-torsionfree (Lemma~\ref{line-prop-lem}(1)), $\Omega =
\{ \dv M | M \text{is a line module for } R\} \subseteq \dv_2(R)$.
But no two single point divisors on $E$ are linearly equivalent, so
$\Omega$ is a countable set of points of $E$. But then $X =
\bigcup_{q \in \Omega} X_q$ expresses $X$ as a countable disjoint
union of closed subvarieties. Since $k$ is uncountable, this is
possible only if the set $\Omega$ is finite.
\end{proof}

\section{The main classification theorem}
\label{main-thm-sec}

The first two lemmas of this section begin to make clear how the
special ideals of the ring $R(D)$ are closely connected to the line
modules of $R(D)$.   We will then use the fact that the divisors of
line modules of $R(D)$ are very restricted to show, in
Theorem~\ref{min-ideal-thm} below, that $R(D)$ has a minimal special
ideal, completing the proof of Theorem~\ref{RD-props-thm}. Our main
classification result, Theorem~\ref{deg3-class-thm}, will then
quickly follow from our previous results in
Section~\ref{subrings-sec}.

For the next two lemmas, we maintain the setup of the previous
section. So let $D'$ be an effective divisor on $E$ with $0 \leq
\deg D' \leq 6$, and let $D = D' + p$ for some point $p \in E$. Let
$R = R(D)$ and $R' = R(D')$, and let $L_R$ be the exceptional line
module for the extension $R \subseteq R'$, as in
Definition~\ref{exc-def}.
\begin{lemma}
\label{smallest-tf-lem} There is a unique smallest nonzero submodule
$N \subseteq L$ such that $L/N$ is $g$-torsionfree.
\end{lemma}
\begin{proof}
Recall that $\dv L = \tau(p)$ by Lemma~\ref{exc-line-lem}(1).
Suppose that $L/N$ is $g$-torsionfree for some submodule $0 \neq N
\subseteq L$.  By Lemma~\ref{line-prop-lem}(2), $N \cong L'[-m]$ is
a shift of a line module $L'$ with $\dv L' = \tau^{m+1}(p)$.  By
Theorem~\ref{line-finite-thm}, there can be only finitely many
values of $m$ for which $\tau^{m+1}(p)$ is the divisor of a line
module over $R$.  Thus there is an upper bound on the possible value
of $m$, which is also the multiplicity of the GK-$1$ module $L/N$.

Thus we may pick a submodule $0 \neq N \subseteq L$ such that $L/N$
is $g$-torsionfree and the multiplicity $m$ of $L/N$ is maximal
among all such choices. If $0 \neq N' \subseteq L$ is another
submodule with $L/N'$ $g$-torsionfree, then $L/(N \cap N')$ is also
$g$-torsionfree.  $L/(N \cap N')$ must then also have multiplicity
$m$, so $N \cap N'$ is equal to $N$ in large degree.  By
$g$-torsionfreeness, $N \cap N' = N$ and so $N \subseteq N'$.
\end{proof}

\begin{lemma}
\label{mult-special-lem} There is a special ideal $I$ of $R$ with
the following property:  given any (possibly infinite) direct sum of
shifts of $L$, say $M = \bigoplus L[a_i]$, and $N \in \rgr R$ such
that $N$ is isomorphic to a subfactor of $M$, $\GK N \leq 1$, and
$N$ is $g$-torsionfree, then $I \subseteq \rann_R N$.
\end{lemma}
\begin{proof}
Let $N$ be as in the statement.  Since $N$ is finitely generated, we
can choose submodules $N'' \subseteq N' \subseteq M$ with $N', N''
\in \rgr R$ such that $N \cong N'/N''$.
Then $N'$ is contained in finitely many terms in the direct sum
comprising $M$, so we may replace $M$ if necessary by a finite sum
$M = \bigoplus_{i = 1}^m L[a_i]$ for some shifts $a_i \in \mb{Z}$.
We adjust the choices of $N'$ and $N''$ further.  If $N' \cap L[a_j]
= 0$ for some $j$, then $N'$ embeds in $M = \bigoplus_{i \neq j}
L[a_i]$, so we might as well replace $M$ by a direct sum of fewer
terms. Thus we can assume that $N' \cap L[a_i] \neq 0$ for all $1
\leq i \leq m$.  If $N'' \cap L[a_i] = 0$ for some $i$, then the
nonzero submodule $N' \cap L[a_i]$ of $L[a_i]$ embeds in $N$. But
$L[a_i]$ is GK-2-critical and so $\GK (N' \cap L[a_i]) = 2$,
contradicting $\GK N = 1$; thus $N'' \cap L[a_i] \neq 0$ for all
$i$.  Now for each $i$, consider the $g$-torsion submodule
$\ts(L[a_i]/(N'' \cap L[a_i])) = P_i/(N'' \cap L[a_i])$, where $P_i
\subseteq L[a_i]$, and let $P = \bigoplus P_i \subseteq M$. Since
$N$ is $g$-torsionfree, $N \cong (N' + P)/(N'' + P)$, so $N$ is
isomorphic to a subfactor of $M/P$. But now $L[a_i]/P_i$ is a
$g$-torsionfree proper factor module of $L[a_i]$. Let $\wt{L}$ be
the unique minimal nonzero submodule of $L$ such that $L/\wt{L}$ is
$g$-torsionfree, as constructed in Lemma~\ref{smallest-tf-lem}; then
$\wt{L}[a_i] \subseteq P_i$. Thus $N$ is a subfactor of $\bigoplus
L[a_i]/\wt{L}[a_i]$.  By Lemma~\ref{special-ideal-lem}(3), $I =
\rann_R L/\wt{L}$ is a special ideal, and clearly $I$ kills $N$.
\end{proof}

We are now about ready to show, in the following theorem, that
$R(D)$ always has a minimal special ideal. First, we recall the
following fact, which is essentially proved in \cite{ATV2}.
\begin{lemma}
\label{T-ideal-lem} The generic Sklyanin algebra $S$ and its
Veronese ring $T = S^{(3)}$ have no homogeneous factor rings of
GK-dimension $1$.
\end{lemma}
\begin{proof}
Looking at the localization $\Lambda = T[g^{-1}] = S[g^{-1}]^{(3)}$,
the degree-$0$ piece $\Lambda_0$ is a simple ring \cite[Corollary
7.9]{ATV2}.  Then if $T$ has a homogeneous factor ring $T/I$ of GK
1, it has such a factor ring $T/I'$ which is also $g$-torsionfree,
by Lemma~\ref{special-ideal-lem}(2); in this case,
$(T/I')[g^{-1}]_0$ is a proper nonzero factor ring of $\Lambda_0$, a
contradiction. The result for $S$ follows the same way.
\end{proof}

\begin{theorem} (Theorem~\ref{RD-props-thm}(3))
\label{min-ideal-thm} Let $R = R(D)$ for any effective $D$ with $0
\leq \deg D \leq 7$.  Then $R$ has a minimal special ideal.
\end{theorem}
\begin{proof}
(1) The proof is by induction on $\deg D$.  For the base case, $D =
0$ and $R(D) = T$, so the result is trivial by
Lemma~\ref{T-ideal-lem} ($T$ itself is minimal in tails.)

Now assume that $\deg D \geq 1$, let $D = D' + p$ for some $p$, and
suppose that we have proven the theorem for $R'= R(D')$.  Consider a
special ideal $J$ of $R$.  We have a chain of special $R$-ideals $J
\subseteq R'J \cap R \subseteq R'J R' \cap R$.  It is easy to see
that $M = (R'J R' \cap R)/(R'J \cap R)$ is isomorphic, as a right
$R$-module, to a subfactor of some direct sum of copies of shifts of
$(R'/R)_R$.   Also, by Lemma~\ref{exc-line-lem}(2), each $(R'/R)_R$
is isomorphic to a direct sum of copies of shifts of the exceptional
line module $L$ for $R \subseteq R'$, as defined in
Definition~\ref{exc-def}. Since $R/(R'J \cap R)$ is a special factor
ring, $\dim_k \ts(M) < \infty$ by Lemma~\ref{special-ideal-lem}(2).
Thus if $M' = M/\ts(M)$, then $M'$ is a $g$-torsionfree,
GK-dimension $1$ subfactor of a direct sum of shifts of $L$.  If $I$
is the  special ideal constructed in Lemma~\ref{mult-special-lem},
we conclude that $M'I = 0$.  Thus $M I_{\geq m_1} = 0$ for some
$m_1$.

Similarly, $N = (R'J \cap R)/J$ is isomorphic as a left $R$-module
to a subfactor of some direct sum of copies of shifts of $_R
(R'/R)$.  Because of Theorem~\ref{Rop-thm}(3), all right-sided
results in this paper have left-sided analogs.  A left-sided version
of the argument in the last paragraph then constructs a special
ideal $H$ of $R$ satisfying $H_{\geq m_2}N = 0$ for some $m_2$.

Now $R'JR'$ is clearly a special ideal of $R'$ (otherwise by
Lemma~\ref{special-ideal-lem}(1), $R'JR' \subseteq R'g$ and so $J
\subseteq R'g \cap R = Rg$, a contradiction).
By the induction hypothesis, $R'$ has a special ideal $K$ which is
minimal in tails, and so $K_{\geq m_3} \subseteq R'JR'$ for some
$m_3$. Then $H_{\geq m_2} (K_{\geq m_3} \cap R)I_{\geq m_1}
\subseteq J$. This implies that $[H(K \cap R)I]_{\geq m_4} \subseteq
J$, for some $m_4$.  Since $G = H(K \cap R)I$ is a special ideal of
$R$ which is independent of $J$, we see that $G$ is an ideal which
is minimal in tails for $R$.
\end{proof}

At this point, we have done all of the hard work necessary to prove
the main classification theorem of the paper,
Theorem~\ref{deg3-class-thm}, and we now put all of the pieces
together.

\emph{Proof of Theorem~\ref{deg3-class-thm}.}  Recall the setup of
the theorem: we have $V \subseteq T_1$ and $A = k \langle V \rangle
\subseteq T$, and we assume that $Q_{\rm gr}(A) = Q_{\rm gr}(T)$. As
usual, let an overline indicate the image of subsets under the
homomorphism $T \to T/Tg \cong B(E, \mc{M}, \tau)$.  Now
$\overline{V}$ generates some subsheaf $\mc{N}$ of $\mc{M}$ on $E$.
Then $\mc{N} = \mc{I}_D \otimes \mc{M}$ for some divisor $D$ on $E$,
and $A \subseteq R(D)$. We claim that $0 \leq \deg D \leq 7$, or
equivalently that $\deg \mc{N} \geq 2$. If not, then $\HB^0(E,
\mc{N}) \leq 1$ by Riemann-Roch, and thus either $V = kg$ or $V = kx
+ kg$ for some $x \in T_1$.  In either case, since $g$ is central,
clearly $A$ is commutative, and so $Q_{\rm gr}(A) = Q_{\rm gr}(T)$
cannot possibly hold; this contradiction proves the claim.

Now Lemma~\ref{B-subring-lem}(3) makes clear that $\overline{A}$ is
equal in large degree to some $B(F, \mc{N}', \tau') \subseteq B(E,
\mc{N}, \tau)$, where $F$ may be a different elliptic curve, but in
any case $\overline{A} \subseteq \overline{R(D)} = B(E, \mc{N},
\tau)$ is a finite ring extension.  Thus Hypothesis~\ref{subs-hyp}
holds, and the results of Section~\ref{subrings-sec} apply. Since
$R(D)$ is now known to have a minimal special ideal by
Theorem~\ref{min-ideal-thm}, Theorem~\ref{subring-thm}(1) shows that
$A$ and $R(D)$ are equivalent orders.  Since the rings $R(D)$ are
maximal orders by Theorem~\ref{max-order-thm}, it is now clear that
the rings $R(D)$ are the only subrings of $T$ generated in degree 1
which have the same graded quotient ring as $T$ and which are
maximal orders.  Clearly a given $A$ is an equivalent order only to
the smallest $R(D)$ containing it, so the $D$ with this property is
unique.

Theorem~\ref{subring-thm}(2) also implies that $A \subseteq R(D)$ is
a finite ring extension if $A$ is noetherian; in particular, this
does hold if $g \in A_1$, by Lemma~\ref{C-noeth-lem}(1).  \hfill
$\Box$

\section{Examples}
\label{example-sec}

Some of the complicated (and interesting) behavior of the rings
$R(D)$ happens when $D$ contains multiple points on the same $\tau =
\sigma^3$-orbit. In this section, we work out some of the features
of one of the simplest such cases.  This will provide explicit
examples of a number of phenomena studied in previous sections.  In
this section and the next, we again need the notation for point
spaces $W(p) \subseteq S_1$, as in  Section~\ref{sklyanin-sec}.

The following is the example that will occupy us for most of this
section.
\begin{example}
\label{main-ex} Fix any $p \in E$, let $D = p + \sigma^{-3}(p) \in
\Div E$, and set $R = R(D)$.  It will be interesting to compare the
two different ways in which $R(D)$ can be thought of as an iterated
blowup of $T$, so we need some more notation.  Set $R' = R(p)$ and
$R'' = R(\sigma^{-3}(p))$.  The first sequence of blowups is $R =
R(D) \subseteq R' = R(p) \subseteq T$. Let $J = \{ x \in R | R'_1x
\subseteq R \}$, so that $L = R/J$ is the exceptional line module
for the blowup $R \subseteq R'$ as in Definition~\ref{exc-def};
similarly, let $J' = \{ x \in R' | T_1x \subseteq R' \}$,  so that
$L' = R'/J'$ is the exceptional line module of $R' \subseteq T$. The
other sequence of blowups is $R = R(D) \subseteq R'' =
R(\sigma^{-3}(p)) \subseteq R$. Let $J^{\circ} = \{ x \in R | R''_1x
\subseteq R \}$, so that $L^{\circ} = R/J^{\circ}$ is the
exceptional line module for $R \subseteq R''$; and let $J'' = \{x
\in R'' | T_1x \subseteq R'' \}$, so that $L'' = R''/J''$ is the
exceptional line module of $R'' \subseteq T$.
\end{example}

\begin{proposition}
\label{main-ex-prop} Example~\ref{main-ex} has the following
properties.
\begin{enumerate}
\item Let $V = W(p)W(\sigma^{-2}(p))S_1$, thought of as a subspace of $R_1$.  Then $I = V R$ is a special ideal
 of $R$ such that $R/I$ is a $g$-torsionfree point module.
\item The subring $A = k \langle V \rangle$ of $R$ is equal to $k + I$,
and $A$ is neither right nor left noetherian.
\item $J \subseteq I$, and $(I/J)_R$ is isomorphic to the shifted line
module $R/J^{\circ}[-1]$.
\item $L'_R \cong L^{\circ}_R$.  In particular, $L'_{R'} \in \rgr R'$ is a GK-2 module such that $L'_R \in \wt{\mc{C}}$
and $L'_{R'} \in \wt{\mc{D}}$, in the notation of
Theorem~\ref{qgr-blowup-thm}.
\end{enumerate}
\end{proposition}

Before proving the proposition, we make some remarks on the
significance of these properties.  Part (1) gives an explicit
example of a special ideal in a ring $R(D)$.  Part (2) shows that
degree-$1$ generated subrings $A$ of $T$ which do not contain the
central element $g$ can indeed be non-noetherian; also,
$A \subseteq R$ need not be a finite ring extension.  Part (3) gives
an explicit example where an exceptional line module (in this case
$L \cong R/J$) contains another shifted line module (in this case
$L^{\circ}[-1]$, a shift of the exceptional module for the other
blowup sequence), where the factor is GK-1 $g$-torsionfree. This is
exactly the phenomenon that Lemma~\ref{smallest-tf-lem} is designed
to control. Finally, part (4) shows, as was claimed earlier, that
GK-2 modules are sometimes among those that need to be quotiented
out from $\rQgr R'$ in the equivalence of categories $\rQgr R/\mc{C}
\sim \rQgr R'/\mc{D}$ in Theorem~\ref{qgr-blowup-thm}.

\begin{proof}[Proof of Proposition~\ref{main-ex-prop}]

(1)  We have $V = W(p)W(\sigma^{-2}(p))S_1 \subseteq S_3 = T_1$.  By
Lemma~\ref{ps-less-basic-lem}(1), we know that $\dim_k V = 7$ and $g
\not \in V$.

An easy calculation, ``moving point spaces" using
Lemma~\ref{ps-basic-lem}(1), shows that $T_1 V = R''_2$, so that $V
\subseteq J''_1$. The Hilbert series of $J''$ is easily determined,
since the Hilbert series of $R''$ is known and $R''/J''$ is a line
module, to be $\D \hs_{J''}(t) = \frac{t^2 + 7t}{(1-t)^3}$. In
particular, we see from this that $V = J''_1$, since $\dim_k J''_1 =
7$ also.  We claim that $J'' = VR''$, so that $J''$ is generated in
degree $1$ as a right $R''$-ideal. Since $R''/J''$ is
$g$-torsionfree, by Lemma~\ref{line-prop-lem}(1), we have $J'' \cap
R''g = J''g$.  Then $\D \hs_{\overline{J''}}(t) = \frac{t^2 +
7t}{(1-t)^2}$. Also $VR'' \subseteq J''$, and the Hilbert series of
$\overline{V} \overline{R''}$ is easily calculated, using
Lemma~\ref{inv-section-lem}, to be the same as that of
$\overline{J''}$, so that $\overline{V} \overline{R''} =
\overline{J''}$. Then we have $h_{VR''}(t) \geq
h_{\overline{VR''}}(t)/(1-t) = h_{\overline{J''}}/(1-t) =
h_{J''}(t)$, forcing $VR'' = J''$, which proves the claim.

It is easy to see that $V \subseteq R_1$, where $\dim_k R_1 = 8$, so
clearly we have $V + kg = R_1$.  Consider $V^2 =
W(p)W(\sigma^{-2}(p))S_1 W(p)W(\sigma^{-2}(p))S_1 $. By moving point
spaces, we see that $W(\sigma^{-2}(p))S_1 W(p)
= W(\sigma^{-2}(p))W(\sigma^{-1}(p))S_1$, and this contains $g$ by
Lemma~\ref{ps-less-basic-lem}(1).  Thus $V^2 \supseteq W(p)g
W(\sigma^{-2}(p))S_1 = Vg$.   This implies that $R_1 V = (V + kg)V =
V^2$, and similarly $V R_1 = V^2$.  Also, then $V^2 + kg^2 = R_1 V +
kg^2 = R_1 V + Vg + kg^2 = R_1V + R_1g = R_1(V + kg) = R_2$. An
inductive argument gives $V^n + kg^n = R_n$ and $VR_{n-1} = V^n =
R_{n-1}V$ for all $n \geq 1$.  Finally, note that $g^n \not \in V^n$
for all $n \geq 1$, because $V^n \subseteq VR'' = J''$, and $g^n
\not \in J''$ for all $n \geq 1$ since $R''/J''$ is $g$-torsionfree.
In conclusion, $I = VR = RV$ is an ideal of $R$ such that $\{1, g,
g^2, \dots \}$ form a $k$-basis for $R/I$. So $I$ is a special ideal
of $R$, and moreover $R/I$ is a right $g$-torsionfree $R$-point
module.

(2) The calculation in (1) also shows that $A = k \langle V \rangle
= k + I$. Clearly $A$ is not right noetherian, since $R_A$ is
infinitely generated, but
$R_A$ embeds as a submodule of $A$ by taking any $0 \neq x \in I$
and considering $xR \subseteq I \subseteq A$.  Similarly, $A$ is not
left noetherian.

(3)  Note that $J'' \subseteq R'$. This is because by part (1), we
have  $J'' = W(p)W(\sigma^{-2}(p))S_1 R''$, which is easily checked
to be contained in $R'$ by moving point spaces (in fact, we even
have $W(p)W(\sigma^{-2}(p))S_1 T_1 = R'_2$.) So $J'' \subseteq R'
\cap R''$. Since $R''/J''$ is a line module,
$\overline{R''}/\overline{J''}$ is a point module, in particular
each graded piece has dimension $1$. By considering vanishing
conditions, it is easy to see that $\overline{(R'' \cap R')_n}
\subsetneq \overline{R''_n}$ for $n \geq 1$, and so this forces
$\overline{R'' \cap R'}/\overline{J''} = k$.  Since $(R'' \cap
R')/J''$ is $g$-torsionfree, being an $R$-submodule of $R''/J''$, we
conclude that $(R'' \cap R')/J''$ is a $g$-torsionfree point module,
which must have  $\{1, g, g^2, \dots \}$ as a $k$-basis.  Then the
natural map $R/(J'' \cap R) \to (R' \cap R'')/J''$ is an isomorphism
of $R$-modules. Clearly $I = VR \subseteq VR'' = J''$. Thus $I
\subseteq J'' \cap R$ and so we have $I = J'' \cap R$ by Hilbert
functions.

Now it is straightforward to check that $J \subseteq J''$, as
follows.  If $x \in J$, then $R'_1 x \subseteq R$. Choosing any $z
\in R''_1 \setminus R'_1$, we have $R'_1 + kz = T_1$ and so $T_1x =
(R'_1x + kzx) \subseteq R + R''_1R \subseteq R''$. Thus $x \in J''$.
So $J \subseteq J'' \cap R = I$.

It is known by the proof of part (1) that $I = VR $ is generated in
degree $1$ as a right $R$-ideal.  So $I/J$ is a cyclic module,
generated in degree $1$, with the Hilbert series of a line module
shifted by $-1$. Examine $X = V J^{\circ}$.  By moving point spaces,
it is easy to see that $R'_1 V = V R''_1$.  Thus $X$ satisfies $R'_1
X = R'_1 V J^{\circ} =VR''J^{\circ} \subseteq V R \subseteq R$. So
$X \subseteq J$. This implies that $J^{\circ} \subseteq \rann(y)$,
where $0 \neq y \in V/J_1$ is a generator of $I/J$. Thus there is a
surjection $R/J^{\circ}[-1] \to I/J$.  This is a surjection between
$R$-modules with the same Hilbert function, so it is an isomorphism.

(4)  Consider the $R$-module $R/(J' \cap R)$.  This module is
certainly $g$-torsionfree (being an $R$-submodule of the line module
$R'/J'$). Moreover, by considering vanishing conditions one easily
sees that $\overline{R}/\overline{J' \cap R}$ has the Hilbert series
of a point module at the very smallest, as follows: since we showed
that $J'' = VR''$ in part (1), an analogous argument gives that $J'
= J'_1 R'$, where $J'_1 = W(\sigma^3(p))W(\sigma(p))S_1$; thus, for
all $n \geq 0$ every element of $\overline{J'_n}$ vanishes at
$\sigma^3(p)$, which is not true of every element of
$\overline{R_n}$.  Then $R/(J' \cap R)$ has the Hilbert series of a
line module at the very smallest. Now we have $J^{\circ} \subseteq
J'$, by a completely analogous proof to the one showing $J \subseteq
J''$ in part (3) (just switch the roles of $R'$ and $R''$.)  Since
$R/J^{\circ}$ is a line module and $J^{\circ} \subseteq (J' \cap
R)$, we conclude that $J^{\circ} = J' \cap R$.  Thus $R/J^{\circ} =
R/(J' \cap R) \to R'/J'$ is an injection of line modules and so is
in fact an isomorphism.

Combined with (3), this shows that $L'_R$ is isomorphic to a shift
of a submodule of $R/J = L$, so that $L'_R \in \wt{\mc{C}}$ is in
the exceptional category of the blowup $R \subseteq R'$.  Then
$L'_{R'} \in \wt{\mc{D}}$ by definition.
\end{proof}

Having a divisor $D$ which contains two points on the same
$\sigma^3$-orbit is not the only situation in which one expects the
ring $R(D)$ to have special ideals.  Such ideals also occur when the
points of $D$ are not in general position in $\mb{P}^2$, for example
if $D$ is the sum of three collinear points, as follows.
\begin{example}
Let $0 \neq f \in S_1$, and let $D = p + q + r$ be the hyperplane
section of $E$ where $f \in \HB^0(E, \mc{L})$ vanishes.  Let $R =
R(D)$. Let $V = fS_2$, so that $\dim_k V = 6$.  We have $\overline{V} =
\overline{R}$, both being the set of sections of $\mc{L}_3$
vanishing along $D$. Since $\dim_k R_1 = 7$, we have $g \not \in V$
and $V + kg = R_1$. Consider $V^2 = fS_2fS_2$.  The space $S_1fS_1$
contains $g$ by Lemma~\ref{ps-less-basic-lem}(3), and so $fS_2fS_2
\supseteq fS_1gS_1 = Vg$.  It then follows, similarly as in the
proof of part (1) of the preceding proposition, that $I = R V= V R$
is a special ideal of $R$ with $V^n = I_n$ and $V^n + kg^n = R_n$
for all $n \geq 1$.  Moreover, since $Sg$ is a completely prime
ideal of $S$, it is easy to check that $g^n \not \in fS$, so that $g^n
\not \in V^n$ for all $n \geq 0$.  Thus $R/I$ is a factor ring which
is also a $g$-torsionfree point module.
\end{example}

%

\section{Subrings of $S$ generated in degree $1$}
\label{deg1-sec}

In this final section, we study the behavior of subrings of a
generic Sklyanin algebra $S$ which are generated in degree $1$.
Because $\dim_k S_1 = 3$, it is not surprising that there are rather
few possibilities.  Thus we can give a much more specific
classification result than in the case of degree-3-generated
algebras, but at the expense of some picky calculations. Still, we
offer this result as a step towards a general theory of subalgebras
generated in an arbitrary degree of $S$. Throughout this section,
let $A = k \langle V \rangle$, where $0 \neq V \subsetneq S_1$.
Certainly if $\dim_k V = 1$, then $A \cong k[z]$, and this case is
both boring and does not have $Q_{\gra}(A) \neq Q_{\gra}(S)$. So the
only case worth studying is $\dim_k V = 2$, and we will see in
Theorem~\ref{deg1-gen-thm} below that the classification is a
dichotomy depending on whether or not $V$ is a point space.

The major technical part of the work is contained in the next lemma.
Though what $\overline{A}$ is equal to in large degree follows
easily from previous lemmas, in case $V$ is not a point space, our
method of proof in the next theorem requires a more careful analysis
of $\dim_k A_n$ for small $n$.
\begin{lemma}
\label{deg3-case-lem} Let $V \subseteq \HB^0(E, \mc{L})$ such that
$\dim_k V = 2$ and $V$ generates $\mc{L}$, and let $A = k \langle V
\rangle \subseteq S$, so that $\overline{A} = k \langle \overline{V}
\rangle \subseteq B = B(E, \mc{L}, \sigma)$.
\begin{enumerate}
\item $\overline{A}$ is equal to $B$ in large degree.

\item $A_1 S_1 = S_2 = S_1 A_1$.

\item $\dim_k A_2 = 4$ and $\dim_k A_3 \geq 7$.
\end{enumerate}
\end{lemma}
\begin{proof}
Throughout the proof of this lemma, we identify $A_{\leq 2}$ with
$\overline{A}_{\leq 2}$.

(1) By Lemma~\ref{B-subring-lem}(3), $\overline{A}$ is equal in
large degree to some $B(F, \mc{M}, \tau)$.  Here, there is a finite
morphism of nonsingular elliptic curves $\theta: E \to F$; $3 = \deg
\mc{L} = (\deg \theta) (\deg \mc{M})$, where $\mc{M}$ is the sheaf
on $F$ generated by $\overline{A}_1$; and $\deg \mc{M} \geq 2$. The
only possibility is that $\theta$ is the identity, and
$\overline{A}$ is equal in large degree to $B = B(E, \mc{L},
\sigma)$.

(2)  This was actually already shown in the course of the proof of
Lemma~\ref{inv-section-lem}(1). Namely, since $\mc{L} \not \cong
\mc{L}^{\sigma}$, the argument there shows not only that the natural
map
\[
\theta: \HB^0(E, \mc{L}) \otimes \HB^0(E, \mc{L}^{\sigma}) \to
\HB^0(E, \mc{L}_2)
\]
is surjective, but that it is surjective even when restricted to $V
\otimes \HB^0(E, \mc{L}^{\sigma})$ or $\HB^0(E, \mc{L}) \otimes
V^{\sigma}$.

(3) Certainly $\dim_k A_2 \leq 4$. Choosing $z \in S_1 \setminus
A_1$, part (2) shows that $(kz + A_1)A_1 = S_2$, and so $zA_1 + A_2 =
S_2$. Since $\dim_k S_2 = 6$, this forces $\dim_k A_2 = 4$.

Now we prove that $\dim_k A_3 \geq 7$ (we haven't tried to figure
out whether the exact value of $\dim_k A_3$ is $7$ or $8$.) The idea
is to show that for generic choice of a $k$-basis $y, z$ for $A_1$,
we have $\dim_k zA_2 \cap yS_2 = 1$. Certainly, then, $A_3 = yA_2 +
zA_2$ will have dimension at least $2(4) -1 = 7$.

Since $V$ generates $\mc{L}$, a basis for the sections in $V$
determines a morphism $\phi: E \to \mb{P}(V^*) \cong \mb{P}^1$. The
map $\phi$ is the composition of the embedding $i: E \to
\mb{P}(B_1^*) \cong \mb{P}^2$ with a projection map given by
projection away from the point $x \in \mb{P}(B_1^*)$ determined by
$V$, where $x \not \in i(E)$.  Thus the hyperplane sections of
$\phi$ are the degree 3 divisors which are the intersections of
lines in $\mb{P}^2$ through $x$ with $E$. Let $D$ be a Weil divisor
such that $\mc{L} \cong \mc{O}_E(D)$, and let $|D|$ be the complete
linear system of $D$. Let $\mf{d} \subseteq |D|$ be the (incomplete)
linear system of hyperplane sections of $\phi$.  Now suppose that
$A_1 = ky + kz$ and $zu = yv$ in $S_3$, with $u \in A_2, v \in S_2$.
Then $\overline{zu} = \overline{yv}$ in $B_3$.  Now $\overline{y}$
vanishes along some divisor $C \in \mf{d}$ and $\overline{z}$ does
not vanish at any of the points in $C$, because lines through $x$ in
$\mb{P}^2$ are uniquely determined by any point in $E$ they go
through.   So necessarily $\overline{u}$ vanishes along $\sigma(C)$,
in other words $\overline{u} \in \HB^0(E, \mc{I}_{\sigma(C)} \otimes
\mc{L}_2) \cap A_2$. We will show that for a generic choice of $y$
(equivalently a generic choice of $C \in \mf{d}$), vanishing along
$\sigma(C)$ presents $3$ linearly independent conditions to sections
in $A_2$. Then $\dim_k \HB^0(E, \mc{I}_{\sigma(C)} \otimes \mc{L}_2)
\cap A_2 = 1$, and completing $y$ to a basis $\{y,z\}$ we will
conclude that $\dim_k zA_2 \cap yS_2 = 1$ as needed.

We want a more geometric interpretation for the sections in $A_2$. A
basis for the sections in $V^{\sigma} \subseteq \HB^0(E,
\mc{L}^{\sigma})$ determines a map $\phi': E \to
\mb{P}((V^{\sigma})^*) \cong \mb{P}^1$, where $\phi' = \phi \circ
\sigma$. The sections in $A_2 \subseteq \HB^0(E, \mc{L}_2)$
determine a map $\theta: E \to \mb{P}(A_2^*) \cong \mb{P}^3$, and
since we know that in fact $A_2 \cong V \otimes_k V^{\sigma}$ by
dimension count, the map $\theta$ factors through the map $\rho =
\phi \times \phi': E \to \mb{P}(V^*) \times \mb{P}((V^{\sigma})^*)
\cong \mb{P}^1 \times \mb{P}^1$; in other words, $\theta$ is $\rho$
followed by a Segre embedding. This allows us to interpret the
linear system of hyperplane sections of $\theta$ as
$(1,1)$-hyperplane sections of $\rho$. Moreover, we claim that
$\rho$ is generically one-to-one (and so birational onto its image).
For we have $\phi = \pi_1 \circ \rho$, where $\pi_1$ is the first
projection $\mb{P}^1 \times \mb{P}^1 \to \mb{P}^1$.  Thus if $d$ is
the degree of $\rho : E \to \rho(E)$ (the number of points in a
generic fiber of this map), then $d$ divides the degree of $\phi$,
which is $3$. Moreover, $d < 3$, because no fiber of $\phi$ can also
be a fiber of $\phi'$, given that $D \not \sim \sigma^{-1}(D)$. So
$d = 1$.

Now a generic $C \in \mf{d}$ must consist of $3$ distinct points;
otherwise, every line in $\mb{P}^2$ through $x$ will be tangent to
$E$, and this sort of thing happens for plane curves rarely, and
certainly not for an embedded elliptic curve \cite[Theorem
IV.3.9]{Ha}.  Since $\rho$ is also generically one-to-one, choosing
a generic $C \in \mf{d}$, then $\rho(\sigma(C))$ will consist of $3$
distinct points. Moreover, those three points will present $3$
linearly independent conditions to sections of $\mc{O}(1, 1)$ on
$\mb{P}^1 \times \mb{P}^1$, unless $\rho(\sigma(C))$ lies entirely on
a line in one of the two rulings. But this would happen only if
$\sigma(C) \in \mf{d}$ or $\sigma^2(C) \in \mf{d}$, both of which
are impossible.  Thus for a generic choice of $C \in \mf{d}$,
$\sigma(C)$ presents $3$ linearly independent conditions to the
sections of $A_2$, as required.
\end{proof}

\begin{theorem}
\label{deg1-gen-thm} Let $V \subseteq S_1$ with $\dim_k V = 2$ and
let $A = k \langle V \rangle$.
\begin{enumerate}
\item If $V = W(p)$ is a point space, then $A^{(3)} = R(D)$, where $D
= p + \sigma^{-1}(p) + \sigma^{-2}(p)$. In this case, $A \cap Sg =
Ag$ and $\HS_A(t) = \D \frac{t^2+1}{(1-t)^2(1 -t^3)}$.  In
particular, $A$ satisfies Hypothesis~\ref{main-hyp} and thus all of
the theorems in Section~\ref{thcr-factor-sec}.

\item If $V$ is not a point space, then $A$ is equal to $S$ in all large
degrees.
\end{enumerate}
\end{theorem}
\begin{proof}
(1).  To show that $A^{(3)} = R(D)$, it suffices to show that $V^3 =
R(D)_1 = \{ x \in S_3 | \overline{x} \in \HB^0(E, \mc{I}_D \otimes
\mc{L}_3) \}$.  Clearly $\overline{V^3} = \overline{R(D)_1}$, using
Lemma~\ref{inv-section-lem}, so it suffices to show that $g \in
V^3$.  This is shown in Lemma~\ref{ps-less-basic-lem}(2).

Note that $\overline{A} = B(E, \mc{N}, \sigma)$, where $\mc{N} =
\mc{I}_p \otimes \mc{L}$, by Lemma~\ref{inv-section-lem}. Consider
$N^{(i)} = \{x \in S | xg^i \in A \}$ for each $i \geq 0$. Obviously
$A \subseteq N^{(i)}$. Moreover, since $R(D) \cap Tg^i = R(D)g^i$,
we have $N^{(i)}_{3n} = A_{3n}$ for all $n \geq 0$. Consider
$\overline{A} \subseteq \overline{N^{(i)}} \subseteq \overline{S}$.
Since $(\overline{N^{(i)}}/\overline{A})_{3n} = 0$ for all $n \geq
0$, and $\overline{A}$ is generated in degree $1$, we see that
$\overline{N}^{(i)}/\overline{A}$ is a direct limit of graded
$\overline{A}$-modules which are finite-dimensional over $k$. On the
other hand, it is clear since $Q_{\gra}(R(D)) = Q_{\gra}(T)$ that
$Q_{\gra}(\overline{A}) = Q_{\gra}(\overline{S})$, so that
$\overline{A} \subseteq \overline{N^{(i)}}$ is an essential
extension of right $\overline{A}$-modules.  Also,
$\ext_{\overline{A}}^1(k, \overline{A}) = 0$, since $\overline{A}$ is
Cohen-Macaulay by Proposition~\ref{B-prop}(4).  This forces
$\overline{N^{(i)}} = \overline{A}$ for each $i$.
 In other words, $N^{(i)} \subseteq A +
Sg$.  Now note that $N^{(1)} \subseteq A + Sg^0 = S$, and suppose
that $N^{(1)} \subseteq A + Sg^i$ for some $i \geq 0$.  Then for $x
\in N^{(1)}$, writing $x = a + sg^i$ with $a \in A$, $s \in S$, we
see that $xg = ag + sg^{i+1} \in A$ and so $s \in N^{(i+1)}$.  Since
$N^{(i+1)} \subseteq A + Sg$, we have $x \in A + Sg^{i+1}$, so
$N^{(1)} \subseteq A + Sg^{i+1}$. By induction on $i$, $N^{(1)}
\subseteq \bigcap_{i \geq 0} A + Sg^{i} = A$, so $A \cap Sg = Ag$.
The Hilbert series follows immediately since $\HS_{\overline{A}}(t)
= \D \frac{t^2 + 1}{(1-t)^2}$.

(2)  By Lemma~\ref{deg3-case-lem}(1), we have that $\overline{A}$ is
equal to $\overline{S}$ in large degree.  We wish to apply
Proposition~\ref{close-prop}(1). For this we need to check the
hypotheses of that proposition, namely that $A \subseteq S$ is a
finite extension, that $Q_{\rm gr}(A) = Q_{\rm gr}(S)$, and that $(A
\cap Sg)/(A \cap Sg^2) \neq 0$. Supposing we prove these things,
then Proposition~\ref{close-prop}(1) will show that $A$ contains a
special ideal of $S$. But $S$ has no such ideals except those of
finite-$k$-codimension, by Lemma~\ref{T-ideal-lem}.  Thus $A$ will
then equal $S$ in large degree, as required.

The remaining needed facts follow from the other parts of
Lemma~\ref{deg3-case-lem}.  First, the proof of
Lemma~\ref{deg3-case-lem}(3) shows that $S_2 = zA_1 + A_2 = A_1z +
A_2$ where $z \in S_1 \setminus A_1$.  Then by the graded Nakayama
lemma, $S = A + zA = Az + A$. So $A \subseteq S$ is a finite
extension. We also have from Lemma~\ref{deg3-case-lem}(3) that
$\dim_k A_2 = 4$, as well as the rather hard-earned estimate $\dim_k
A_3 \geq 7$.  Since $S_3 = zA_2 + A_3 = A_2z + A_3$, with $\dim_k
S_3 = 10$, this implies that $(zA \cap A)_3 \neq 0$. Choose $0 \neq
x \in A_2$ such that $zx \in A_3$. Since $g \in S_3 = A_2z  + A_3$,
we have $gx \in A_2zx + A_3x \subseteq A_5$.  Thus $0 \neq gx \in (A
\cap Sg) \setminus (A \cap Sg^2)$, so that $(A \cap Sg)/(A \cap Sg^2)
\neq 0$. Also, since we have seen that $zA \cap A \neq 0$, then $z
\in Q_{\rm gr}(A)$. So $S_1 \subseteq Q_{\rm gr}(A)$, and thus $Q_{\rm
gr}(S) = Q_{\rm gr}(A)$. All of the hypotheses of
Proposition~\ref{close-prop}(1) are now verified, and we are done.
\end{proof}

As a prelude to future work, we close with the following example of
a subring of $S$ generated in degree $2$.  It points to a new issue
that arises in trying to understand algebras generated in that
degree, and presumably all other degrees; degrees $1$ and $3$ were
special because they divide the degree of the central element in
$S$.
\begin{example}
Let $p \in E$ be given, and let $A = k \langle V \rangle \subseteq
S^{(2)}$, where $V = W(p)S_1 \subseteq S_2$ is the set of all
sections of $S_2 \cong B(E, \mc{L}_2)$ vanishing along $p$.  Since
we know from the theorems in this paper that both the degree-3
generated algebra $k \langle W(p)S_2 \rangle \cong R(p)$ and the
degree-1-generated algebra $k \langle W(p) \rangle$ are very nice
rings, we might hope $A$ is similarly good.

Now $V^2 = W(p)S_1W(p)S_1 = S_1W(\sigma(p))W(p)S_1 \supseteq S_1g$
by Lemma~\ref{ps-less-basic-lem}(1).  Then $V^3 \supseteq VS_1g$ and
$V^3 \supseteq S_1gV = S_1Vg$.  Since it is easy to prove that $VS_1
+ S_1V = S_3$, we have $V^3 \supseteq S_3g$.  An similar inductive
argument shows that $\bigoplus_{i \geq 0} S_{(2i+1)}g \subseteq A$.
We claim then that $A$ is not left noetherian.  Indeed, if $A$ were
left noetherian, then $S' = \bigoplus_{i \geq 0} S_{2i+1}$ would be
a finitely generated left $A$-module, since $S' \cong S'g$ (as
ungraded modules) and $S'g \subseteq A$.  But then $_{\overline{A}}
\overline{S'}$ would be finitely generated, where $\overline{S'} =
\bigoplus_{i \geq 0} \HB^0(E, \mc{L}_{2i +1})$, and by the graded
Nakayama lemma, we would have to have $\dim_k
\overline{S'}/(\overline{A_{\geq 1}S'}) < \infty$. However, since
$\overline{A} = B(E, \mc{N}, \sigma^2)$ with $\mc{N} = \mc{I}_p
\otimes \mc{L}_2$, we have $(\overline{A_{\geq 1}S'}) \subseteq
\bigoplus_{i \geq 0} \HB^0(E, \mc{I}_p
 \otimes \mc{L}_{2i +1})$ for all $i \geq 0$, from which it is clear
that $\dim_k \overline{S'}/(\overline{A_{\geq 1}S'}) = \infty$.
\end{example}

\providecommand{\bysame}{\leavevmode\hbox to3em{\hrulefill}\thinspace}
\providecommand{\MR}{\relax\ifhmode\unskip\space\fi MR }
\providecommand{\MRhref}[2]{%
  \href{http://www.ams.org/mathscinet-getitem?mr=#1}{#2}
} \providecommand{\href}[2]{#2}

\end{document}